\documentclass[12pt]{amsart}
\usepackage{amssymb,amsmath,amsthm}
\usepackage{latexsym}
\usepackage{epsfig}
\usepackage{graphicx}

\newtheorem{theorem}{Theorem}[section]

\newtheorem{definition}[theorem]{Definition}
\newtheorem{conjecture}[theorem]{Conjecture}
\newtheorem{proposition}[theorem]{Proposition}

\newtheorem{problem}[theorem]{Problem}
\newtheorem{example}[theorem]{Example}

\long\def\symbolfootnote[#1]#2{\begingroup%
\def\thefootnote{\fnsymbol{footnote}}\footnote[#1]{#2}\endgroup}

\newcommand\Z{{\mathbb Z}}

\newcommand\R{{\mathbb{R}}}
\begin{document}
\title{On Slavik Jablan's work on 4-moves}
\author{J\'ozef H. Przytycki} 

\maketitle
\markboth{\hfil{\sc 4-moves}\hfil}
\ \
\centerline{Dedicated to Memory of Slavik Jablan (1952--2015)}

\tableofcontents

{\bf Abstract}
We show that every alternating link of 2-components and 12 crossings can be reduced by 4-moves to the trivial link or the Hopf link. 
It answers the question asked in one of the last papers by Slavik Jablan.
\section{Introduction}

One of the last papers of Slavik Jablan, \cite{DJKS}, concerns the 4-move conjectures by
Nakanishi and Kawauchi \cite{Kir}. It is proved in that paper that the Nakanishi 4-move conjecture holds for knots up to 12 crossings and
Kawauchi 4-move conjecture for 2-component links holds for links up to 11 crossings and alternating links of 12 crossings with
the possible exception of the link of 12 crossings with the Conway symbol $9^*.2:.2:.2$, Figure 1.1 (or its mirror image). 
We show here that the link, which we name $12^2_{Jab}$, can be reduced to the Hopf link by 4-moves\footnote{I was kindly informed by 
Neil Hoffman and Seung Yeop Yang, that the link $12^2_{Jab}$ in SnapPy \cite{CDW} is denoted by $L12a_{1388}$.}.
 Thus every alternating link of two components and $12$ crossings can be reduced to 
the trivial link or the Hopf link by 4-moves, answering a question in \cite{DJKS}. 

We start with the survey of the story of the 4-move conjecture and complete the paper with  
a few remarks that outline the theory of quartic skein modules with a skein relation that is a deformation of a 4-move.  

\vspace*{0.3in} 
\centerline{\psfig{figure=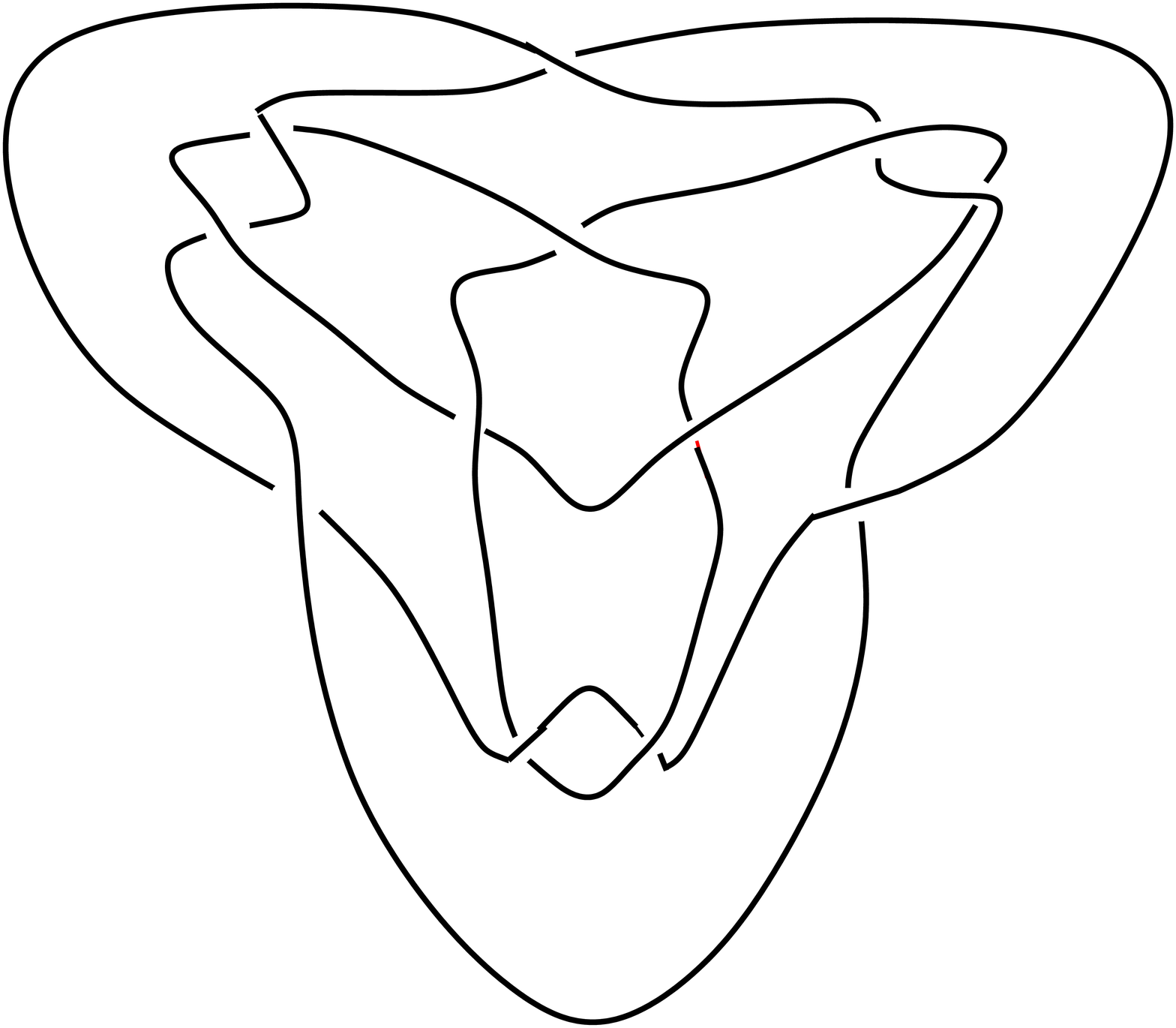,height=5.0cm}}
\begin{center}
Figure 1.1; The link $12^2_{Jab}$ of $12$ crossings, $9^*.2:.2:.2$ in Conway's notation -- to be simplified
\end{center}

\section{History of 4-move problems}
In 1979, Yasutaka Nakanishi, a graduate student at Kobe University\footnote{His advisor was Fujitsugu Hosokawa, born 1930, student of Terasaka, 
a creator of the Kansai school of knot theory. Nakanishi was also a student of 
Shin'ichi Suzuki; see \cite{Prz-5}.} 
conjectured that any knot can be reduced by 4-moves to the trivial knot. He checked the conjecture on small knots and was 
always able to find a reduction.  
Not every link can be reduced to a trivial link by 4-moves. In particular, the linking matrix modulo 2 is preserved by 4-moves. Furthermore,
Nakanishi and Suzuki, using Alexander matrices, demonstrated that the Borromean rings cannot
be reduced to the trivial link of three components \cite{N-S,Nak-4}.
In 1985, after the seminar talk given by Nakanishi in Osaka (likely  a KOOK seminar), there
was discussion about possible generalization of the Nakanishi 4-move
conjecture for links. Akio Kawauchi formulated the question for links: are they equivalent by 4-moves if they are link 
homotopic\footnote{Recall that two links $L_{1}$ and $L_{2}$ are link homotopic if $L_{2}$ can
be obtained from $L_{1}$ by a finite sequence of self-crossing changes.}?

In the following subsections I introduce the 4-move conjectures formally in the context I learned about them at the Braid conference in July, 1986.

\subsection{Braids 1986, Santa Cruz conference}
\ \\

I learned about the 4-move conjectures at the AMS Conference on ``Artin's Braid Group" at the University of California, Santa Cruz, July 1986, 
from Kunio Murasugi and Hitoshi Murakami. 
I just had come from Poland to Vancouver (as a postdoc of Dale Rolfsen) and a few days later took a Greyhound bus to California 
(Santa Cruz).  In Poland in the Spring of 1986, I was working on $t_k$-moves on links and 
so I was immediately hooked by the 4-move (and analogous 3-move) 
conjectures, see Figure 2.1. In the Proceedings of the Conference I wrote about them twice , once in my paper on $t_k$-moves \cite{Prz-1} and 
then, in almost identical words in the Problem list prepared by Hugh Morton \cite{Mor}. Below I cite the Problem B from this list.

\vspace*{0.3in} \centerline{\psfig{figure=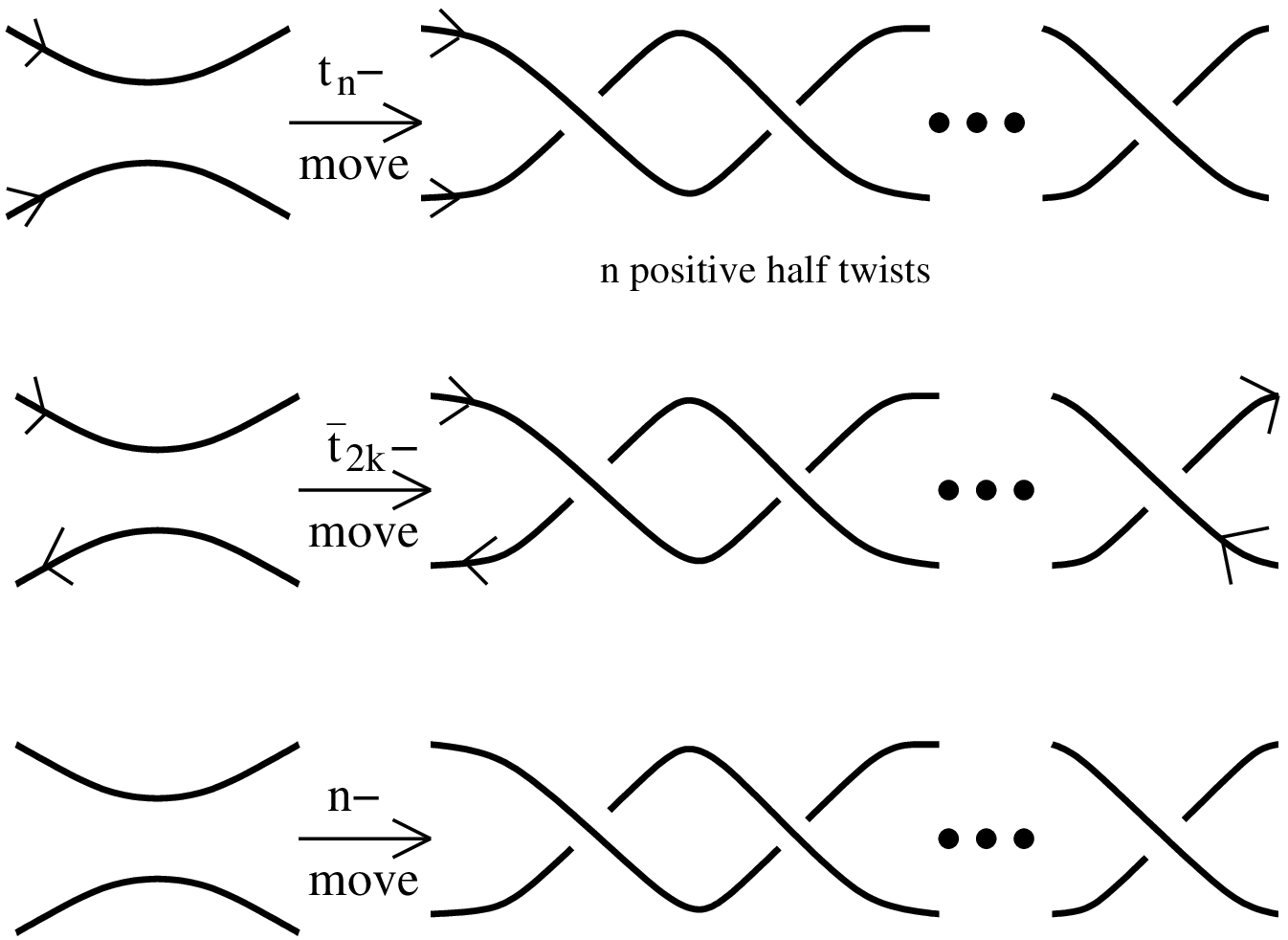,height=5.5cm}}
\begin{center}
Figure 2.1; Oriented $t_n$ and $\bar t_{2k}$-moves and unoriented $n$-move
\end{center}

We say that two links are $n$-equivalent if one can reach one from the other by a finite number of $n$-moves (or their inverses). 
In the case where we only  work with diagrams we also  have to allow  Reidemeister moves.

\begin{conjecture} (Problem B \cite{Mor}, Kawauchi, Nakanishi) If two links are [link] homotopic then they are 4-equivalent. 
In particular, any knot is 4-equivalent to the unknot.
\end{conjecture}
I commented on the conjecture in \cite{Mor}: ``The conjecture is true for 2-bridge\footnote{I am not absolutely sure which argument I used then but 
most likely I used a argument similar to the proof of Proposition 4.3 but with smaller number of the pairs of tangles to check.}
 and pretzel links and for closed 3-braids." 
The 3-braids case is interesting because it was at the Braids conference that I learned (from Orlik via Wajnryb) of the 
Coxeter theorem which says that the braid group $B_3$ divided by fourth powers of generators is finite of $96$ elements \cite{Cox}.
One should add that Coxeter result of 1957 was not well known and was not listed in Mathematical Review.

I have been thinking about 4-move conjectures for almost 30 years so what is below are my, not always fully organized 
thoughts, about approaches to the conjectures. I think I am justified by the fact that both conjectures are still open.

\subsection{Kirby Problem list}

I spent the academic year 1994-5 in Berkeley visiting Vaughan Jones. 
Rob Kirby was then preparing the new ``Problems in Low-Dimensional Topology" list 
and in Fall of 1994 I wrote several problems including Nakanishi and Kawauchi 4-moves conjectures. I wrote:
\begin{conjecture}\label{Conjecture 3}
\begin{enumerate}\ 
\item[(a)] (Nakanishi 79) Any knot is 4-equivalent to the unknot. 
\item[(b)] (Kawauchi 85) If two links are link-homotopic then they are 4-equivalent.
\end{enumerate}
\end{conjecture}
I commented as follows: ``Conjecture \ref{Conjecture 3}(a) holds for algebraic (in the Conway sense) knots and 3-braid knots.
The smallest known unsolved case of Conjecture \ref{Conjecture 3}(a) is a 2-cable knot of the trefoil (see Figure 3.4), which is a 4-braid
knot (Nakanishi 1994 \cite{Nak-3}). See  Nakanishi, 1984 \cite{Nak-1}, Nakanishi \& Suzuki, 1987, \cite{N-S}, Morton, 1988 \cite{Mor}, 
Przytycki, 1988 \cite{Prz-1}."

\begin{example}\label{Example 2.3}
Reduction of the trefoil and the figure eight knots is
illustrated in Figure 2.2.
\end{example}
\centerline{\psfig{figure=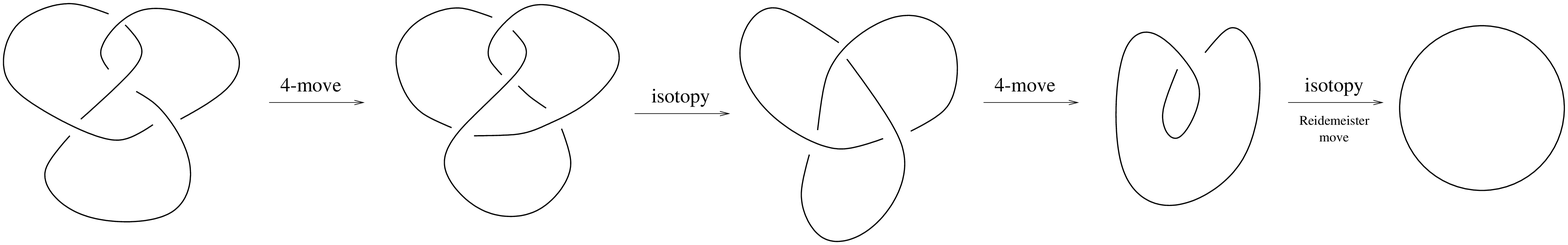,height=2.2cm}}
\centerline{Figure 2.2; basic $4$-move reductions}

In 1998 (after the Knots in Hellas conference) I learned that Nikos Askitas simplified the 2-cable of the torus knot and is proposing the 2-cable of the 
figure eight knot (of 17 crossings) as the simplest counterexample to Nakanishi 4-move conjecture \cite{Ask-1}\footnote{Askitas writes there: 
{\it  In the private conversation with Nakanishi (At the ``Knots in Hellas 98" Conference in Delphi, Greece in August of 1998) this author was told 
of the failure of various invariants in detecting counterexamples.}}. We will go back to Askitas work later 
in the paper but here we include the anecdote: Nikos solved the problem while in military service spending nights guarding some remote outposts and 
manipulating knots. It was also at the Knots in Hellas conference that I met Slavik Jablan for the first time; see \cite{Hel}.

\section{4-move conjectures for algebraic links and 3-braids}\label{Section 3} 
To illustrate early examples of approaches to 4-move conjectures I will show here how they are proved for 2-algebraic knots and links 
of two components (algebraic links in the sense of Conway). I will follow the presentation I gave in the series of talks given at the conference 
in Cuautitlan, Mexico, March 12-16, 2001\footnote{I would like to thank Zbyszek Oziewicz for organizing this amazing conference.}
  \cite{Prz-6,Prz-7,Prz-8}.\footnote{The notes from my talk were supposed to be published 
soon after the conference in the form they were on the arXiv (\cite{Prz-6}). 
Somehow the Proceedings were not published so I divided the paper into two parts and published its extended version 
in \cite{Prz-7} and \cite{Prz-8}. Finally, the original version was published in Spanish \cite{Prz-6}.} The results were first reported in \cite{Kir}.

\begin{definition}[\cite{Co,B-S}]\label{1.6}\ \\
2-algebraic tangles is the smallest family of 2-tangles which satisfies:\\
(0) Any 2-tangle with 0 or 1 crossing is 2-algebraic.\\
(1) If $A$ and $B$ are 2-algebraic tangles
    then $r^i(A)*r^j(B)$ is 2-algebraic; $r$ denotes here
the rotation of a tangle by $90^o$ angle along the $z$-axis, and * denotes
the (horizontal) composition of tangles. For example

\centerline{\psfig{figure=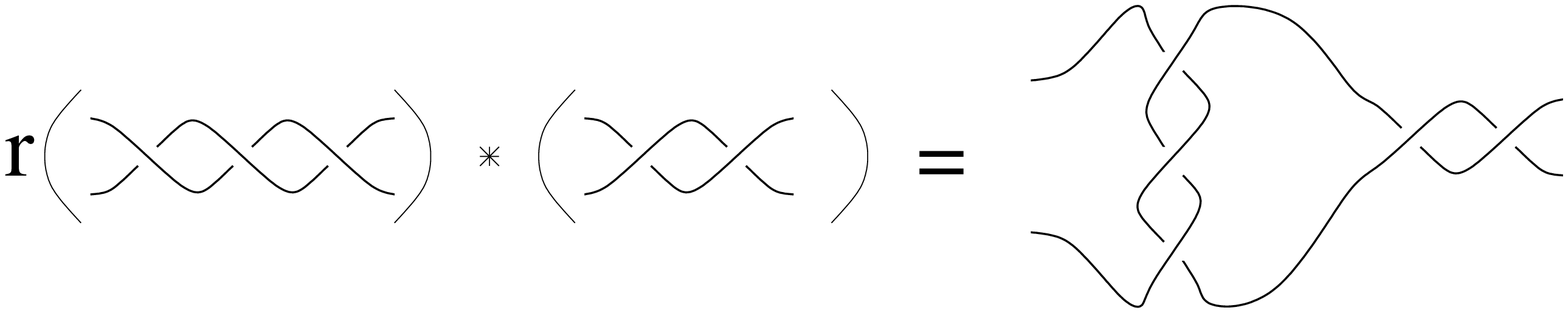,height=2.6cm}}

A link is 2-algebraic if it is obtained from a 2-algebraic tangle
by closing its ends without introducing any new crossings.
\end{definition}


\begin{proposition}[\cite{Prz-6,Prz-7}]\label{1.9}\
\begin{enumerate}
\item[(i)]
Every 2-algebraic tangle without a closed component can
be reduced by $\pm 4$-moves to one of the six basic 2-tangles
shown in Figure 3.1.
\item[(ii)]
Every 2-algebraic knot can be reduced by $\pm 4$-moves to the trivial knot.
\end{enumerate}
\end{proposition}
\ \\
\centerline{\psfig{figure=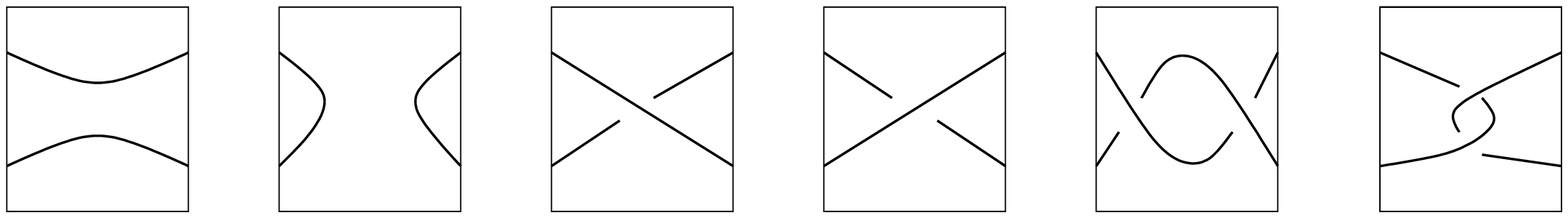,height=1.8cm}}
\centerline{\ \ \  $e_1$\ \ \ \ \ \ \ \ \ \ \ \
 $e_2$ \ \ \ \ \ \ \ \ \ \ \ \
 $e_3$\ \ \ \ \ \ \ \ \ \ \ \ \
$e_4$\ \ \ \ \ \ \ \ \ \ \ \ \ \
$e_5$\ \ \  \ \ \ \ \ \ \ \ \ \ $e_6$\ \ \  }
\centerline{Figure 3.1}
\ \\

\begin{proof}
To prove (i) it suffices to show that every composition (with possible rotation)
of tangles presented in Figure 3.1 can be reduced by $\pm 4$-moves back
to one of the tangles in Figure 3.1 or it has a closed component.
These can be easily verified by inspection. Figure 3.2 is the
multiplication table for basic tangles. We have chosen our basic tangles
to be invariant under the rotation $r$, so it suffices to be able to reduce
every tangle of the table to a basic tangle.
One example of such a reduction is shown in Figure 3.3. \
Part (ii) follows from (i).
\end{proof}
\ \\
\centerline{\psfig{figure=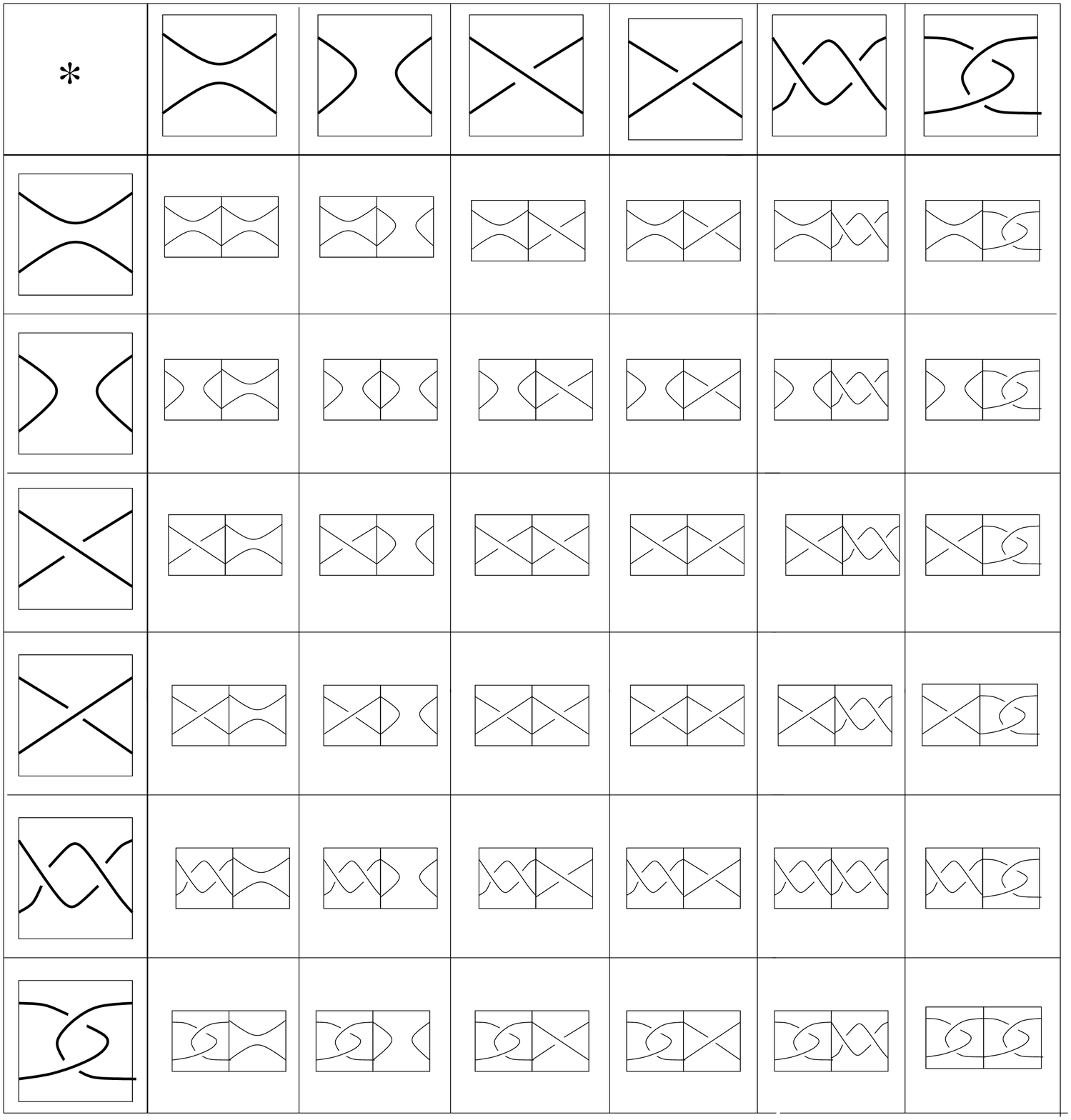,height=8.9cm}}
\centerline{Figure 3.2; tangle multiplication}

\ \\
\centerline{\psfig{figure=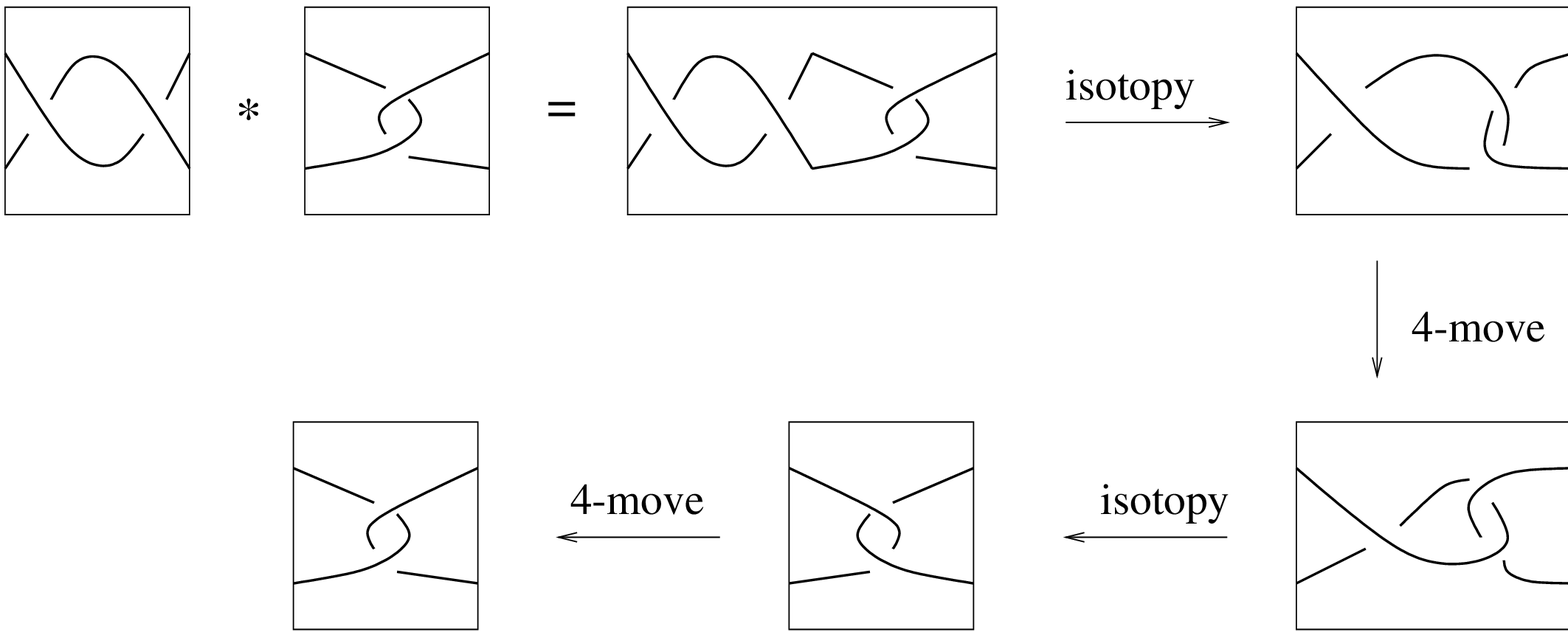,height=3.6cm}}
\centerline{Figure 3.3}
\ \\

As mentioned before, in 1994, Nakanishi became suspicious that a $2$-cable
of the trefoil knot, Figure 3.4, cannot be simplified by 4-moves \cite{Kir}.
However, Nikos Askitas was able to simplify this knot \cite{Ask-1}.

\ \\
\centerline{\psfig{figure=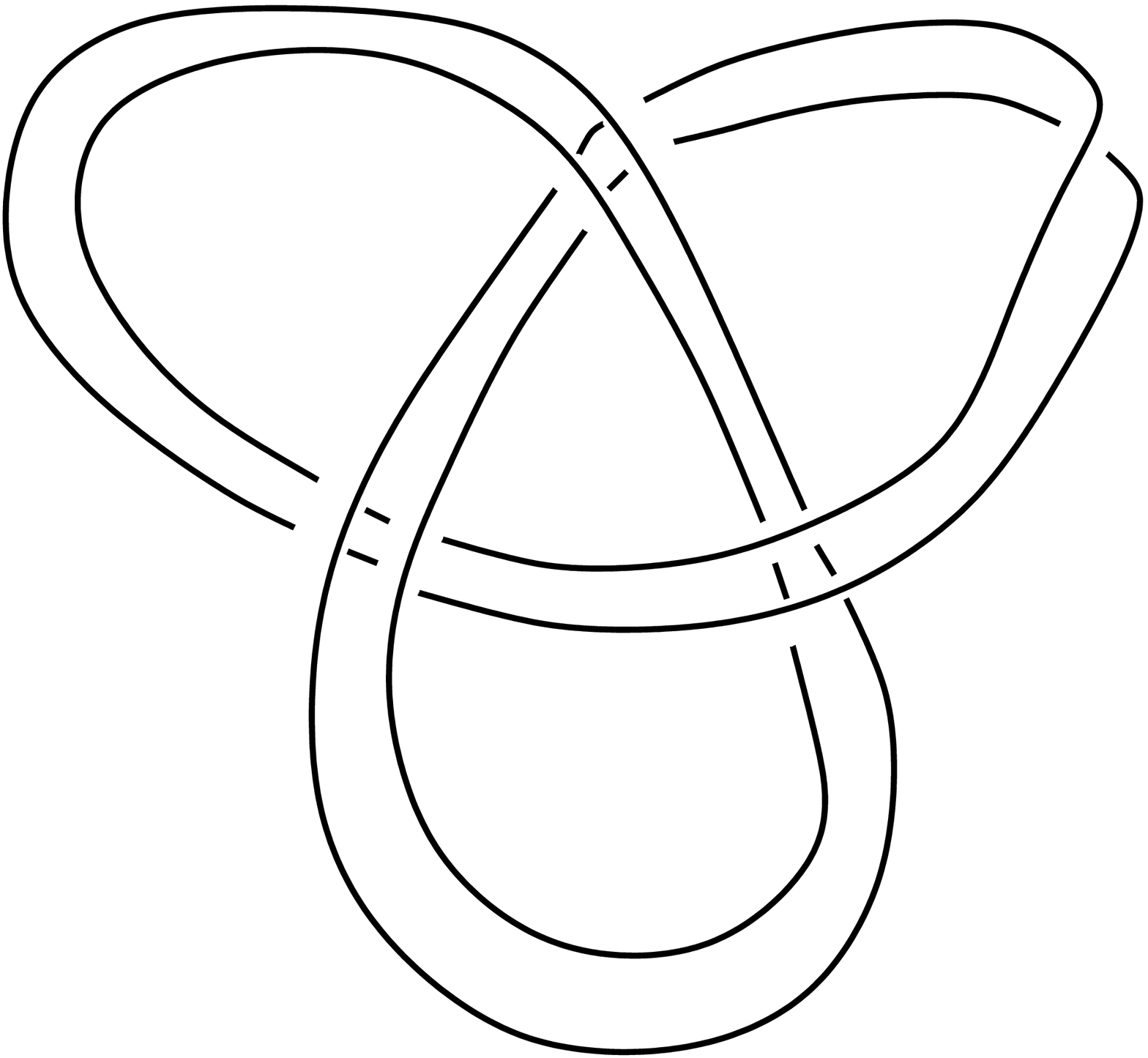,height=4.6cm}}
\centerline{Figure 3.4; 2-cable of the right handed trefoil knot}
\ \\

Askitas, in turn,
suspects that the $(2,1)$-cable of the figure eight knot
(with $17$ crossings) to be the simplest counterexample to the
Nakanishi $4$-move conjecture\footnote{While guarding some remote post, Nikos had time on his hands to draw pictures and apparently 
among them were reductions of 2-cables of $(5,2),(7,2),(9,2)$, and $(11,2)$ torus knots. Rob Todd, with his student Andrew Tew were 
trying to reduce the $2$-cable of the $(5,2)$ torus knot but without a luck 
(the information is given in an undergraduate research project of Andrew Tew).}.

Not every link can be reduced to a trivial link by 4-moves.
In particular, the linking matrix modulo 2 is preserved by $4$-moves.
Furthermore, Nakanishi and Suzuki demonstrated that the
Borromean rings cannot be reduced to the trivial link
of three components \cite{N-S,Nak-4}.

In 1985, after the seminar talk given by Nakanishi in Osaka,
there was discussion about possible
generalizations of the Nakanishi 4-move conjecture for links.
Akio Kawauchi formulated the following question for links:
\begin{problem}[\cite{Kir}]\label{1.10}\
\begin{enumerate}
\item[\textup{(i)}]
Is it true
that if two links are link-homotopic\footnote{Two links $L_1$ and $L_2$ are
{\it link-homotopic} if one can obtain $L_2$ from $L_1$ by a finite
number of crossing changes involving only self-crossings of the components.}
then they are 4-move equivalent?
\item[\textup{(ii)}] In particular, is it true
that every two-component link is 4-move equivalent to the trivial link
of two components or to the Hopf link?
\end{enumerate}
\end{problem}
We can extend the  argument used in Proposition 3.2 to show:
\begin{theorem}\label{1.11}
Any two component 2-algebraic link is 4-move equivalent
to the trivial link of two components or to the Hopf link.
\end{theorem}
\begin{proof}
Let $L$ be a 2-algebraic link of two components.
Therefore, $L$ is built inductively as in Definition 3.1.  Consider the first
tangle, $T$, in the construction, which has a closed component
(if it happens). The complement $T'$ of $T$ in the link $L$ is also
a 2-algebraic tangle but without a closed component. Therefore it
can be reduced to one of the $6$ basic tangles shown in
 Figure 3.1, say $e_i$.
Consider the product $T*e_i$. The only nontrivial tangle $T$ to be
 considered is $e_6*e_6$ (the last tangle in Figure 3.2).
The compositions $(e_6*e_6)*e_i$ are illustrated in Figure 3.5.
The closure of each of these product tangles (the numerator
or the denominator)
 has to have two components because it is $4$-move equivalent to $L$ of no more than two components.
We can easily check that it reduces to the trivial link of two components.
\end{proof}
\ \\

\centerline{\psfig{figure=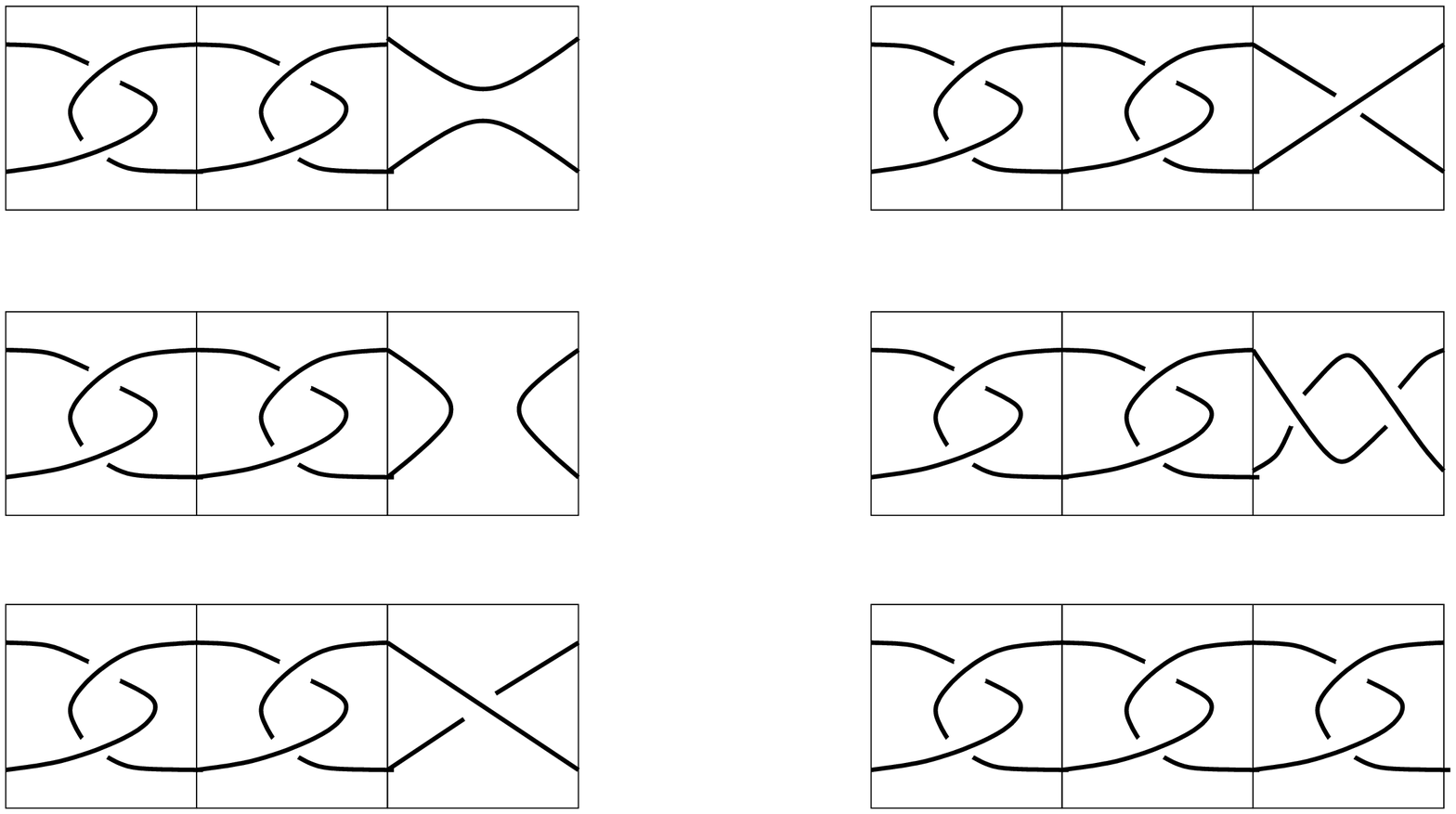,height=4.8cm}}
\centerline{Figure 3.5}
Nakanishi \cite{Nak-2} pointed out that  the ``half"
2-cabling of the Whitehead link, ${\mathcal W}$, Figure 3.6,
was the simplest link which he could not reduce to a trivial link
 by $\pm 4$-moves and is link-homotopic to a trivial
link\footnote{In fact, in  June of 2002
we showed that this example cannot be reduced by $\pm 4$-moves
 \cite{D-P-2}; see Section 4.}.
\ \\
\ \\
\centerline{\psfig{figure=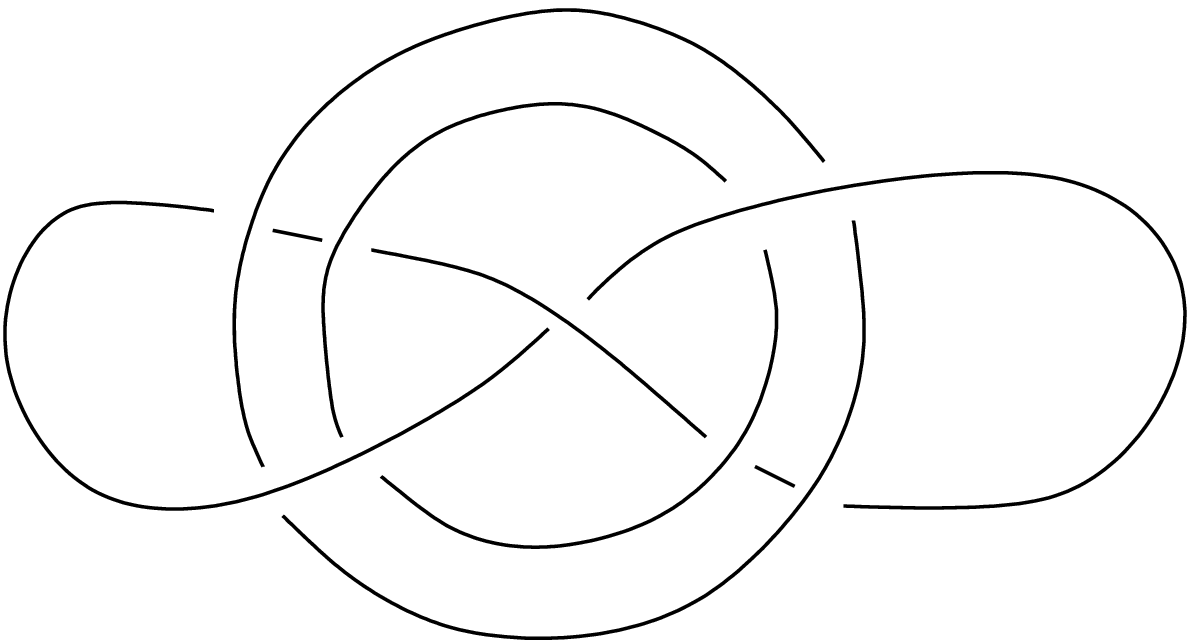,height=4.8cm}}
\centerline{Figure 3.6; the ``half" 2-cabling of the Whitehead link}\ \\
\ \\
It is shown in Figure 3.6 that the link ${\mathcal W}$ is 2-algebraic.
Similarly, the Borromean rings, $BR$, are 2-algebraic,
Figure 3.7.   \\
\ \\
\centerline{\psfig{figure=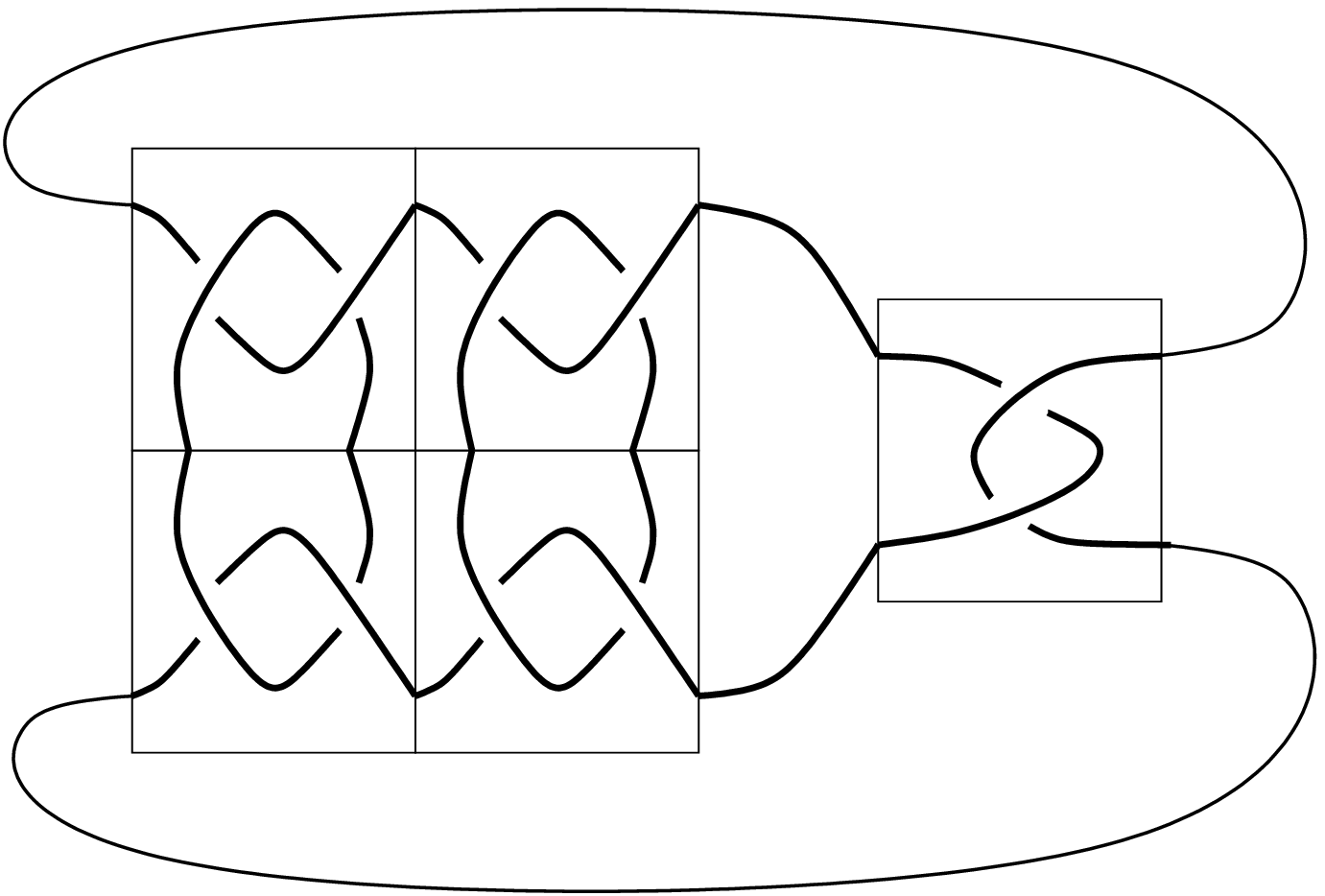,height=5.3cm}}
\centerline{Figure 3.7; ${\mathcal W}=
N(r(r(e_3*e_3)*r(e_4*e_4))*r(r(e_3*e_3)*r(e_4*e_4))*r(e_3*e_3))$
}

\ \\
\centerline{\psfig{figure=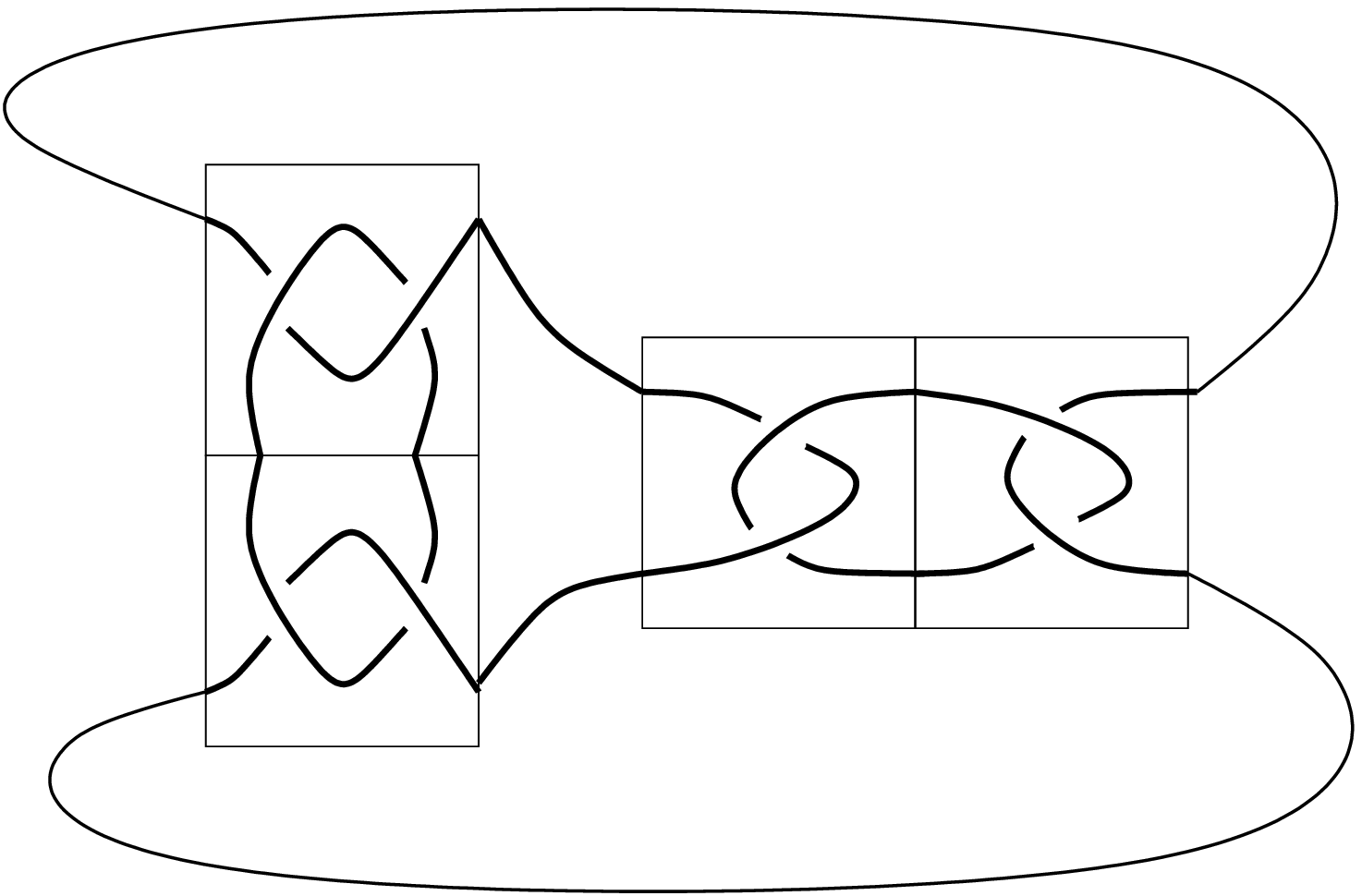,height=5.3cm}}
\centerline{Figure 3.8;
$BR= N(r(r(e_3*e_3)*r(e_4*e_4))*r(e_3*e_3)*r(e_4*e_4))$}
\ \\
\ \\

\begin{problem}\label{1.13}
Is the link ${\mathcal W}$ the only 2-algebraic link of three components (up
to 4-move equivalence) which
is homotopically trivial but which is not 4-move equivalent to the
trivial link of three components?
\end{problem}

We can also prove that the answer to Kawauchi's question is affirmative
for closed 3-braids.
\begin{theorem}\label{1.14}
\begin{enumerate}
\item[(a)] Every knot which is a closed 3-braid is
4-move equivalent to the trivial knot.
\item[(b)] Every link of two components which is
a closed 3-braid is 4-move equivalent to the trivial link of two
 components or to the Hopf link.
\item[(c)]
Every link of three components which is a closed
3-braid is $4$-move equivalent to one of the following:
\begin{enumerate} 
\item[(i)] The trivial link of three components, or 
\item[(ii)] the Hopf link with the additional trivial component, or 
\item[(iii)] the connected sum of two Hopf links, or 
\item[(iv)] the $(3,3)$-torus link, $\bar 6^3_{1}$, represented in the braid group $B_3$
by $(\sigma_1\sigma_2)^3$  (all linking numbers are equal to 1), or 
\item[(v)] the Borromean rings (represented by $(\sigma_1\sigma_2^{-1})^3$).
\end{enumerate}
\end{enumerate}
\end{theorem}
\begin{proof}
Our proof is based on
the Coxeter theorem that the quotient group $B_3/(\sigma_i^4)$ is
finite with $96$ elements, \cite{Cox}.
Furthermore, $B_3/(\sigma_i^4)$ has 16 conjugacy classes\footnote{$Id,\sigma_1,
\sigma_1^{-1},\sigma_1^{2}, \sigma_1\sigma_2,\sigma_1^{-1}\sigma_2,
\sigma_1^{-1}\sigma_2^{-1},
\sigma_1^{2}\sigma_2, \sigma_1^{2}\sigma_2^{-1}, \sigma_1^{2}\sigma_2^{2},
\sigma_1\sigma_2^{-1}\sigma_1\sigma_2^{-1}, 
\sigma_1\sigma_2^{2}\sigma_1\sigma_2^{-1}, \\
\sigma_1\sigma_2^{-1}\sigma_1^{2}\sigma_2^{-1},
\sigma_1\sigma_2^{2}\sigma_1\sigma_2^{2},
\sigma_1^{-1}\sigma_2^{2}\sigma_1^{-1}\sigma_2^{2},
(\sigma_1\sigma_2^{-1})^3$ \ (checked by M.~D{\c a}bkowski
using the {\it GAP} program).}:
9 of them can be  easily identified as representing trivial links (up to
$4$-move equivalence), and
2 of them represent the Hopf link
($\sigma_1^{2}\sigma_2$ and  $\sigma_1^{2}\sigma_2^{-1}$),
and $\sigma_1^{2}$ represents the Hopf link with an additional
trivial component.
We also have the connected sums of
Hopf links ($\sigma_1^{2}\sigma_2^{2}$).
Finally, we are left with two representatives of the
link $\bar 6^3_{1}$ ($\sigma_1\sigma_2^{2}\sigma_1\sigma_2^{2}$ and
$\sigma_1^{-1}\sigma_2^{2}\sigma_1^{-1}\sigma_2^{2}$)
 and the Borromean rings.
\end{proof}
Proposition 3.2 and Theorems 3.4, and 3.6 can be used to analyze
$4$-move equivalence classes of links with small number of crossings.
\begin{theorem}\label{Theorem 3.7}
\begin{enumerate}
\item[\textup{(i)}]
Every knot of no more than $9$ crossings is $4$ move equivalent
to the trivial knot.
\item[\textup{(ii)}] Every two component link of no more than $9$ crossings
is 4-move equivalent to the trivial link of two components or to
the Hopf link.
\end{enumerate}
\end{theorem}
\begin{proof}
Part (ii) follows immediately as the only 2-component links
with up to 9 crossings which are not 2-algebraic are
$9^2_{40}, 9^2_{41}, 9^2_{42}$ and $9^2_{61}$ and all these links
are closed 3-braids.
There are at most $6$ knots with up to 9 crossings which are
neither 2-algebraic nor 3-braid knots. They are:
$9_{34},9_{39}, 9_{40}, 9_{41}, 9_{47}$ and $9_{49}$.
We reduced three of them, $9_{39},9_{41}$ and $9_{49}$ at my Fall 2003
Dean's Seminar. The knot $9_{40}$ was reduced in December of
2003 by Slavik Jablan and Radmila Sazdanovic\footnote{They visited me in Washington participating in Knots in Washington XVII conference, 
December 19-21, 2003.}.
Soon after, my student Maciej Niebrzydowski simplified the remaining
pair $9_{34}$ and $9_{47}$; the last reduction is illustrated in Figure 3.9.
\end{proof}
As we see already in 2003 Slavik Jablan become involved in work on 4-move. Theorem \ref{Theorem 3.7} was first improved by M.~D{\c a}bkowski to 
all links of 10 crossings and then, with the help of Slavik, to links of 11 crossings and knots of 12 crossings \cite{DJKS}.

\ \\
\centerline{\psfig{figure=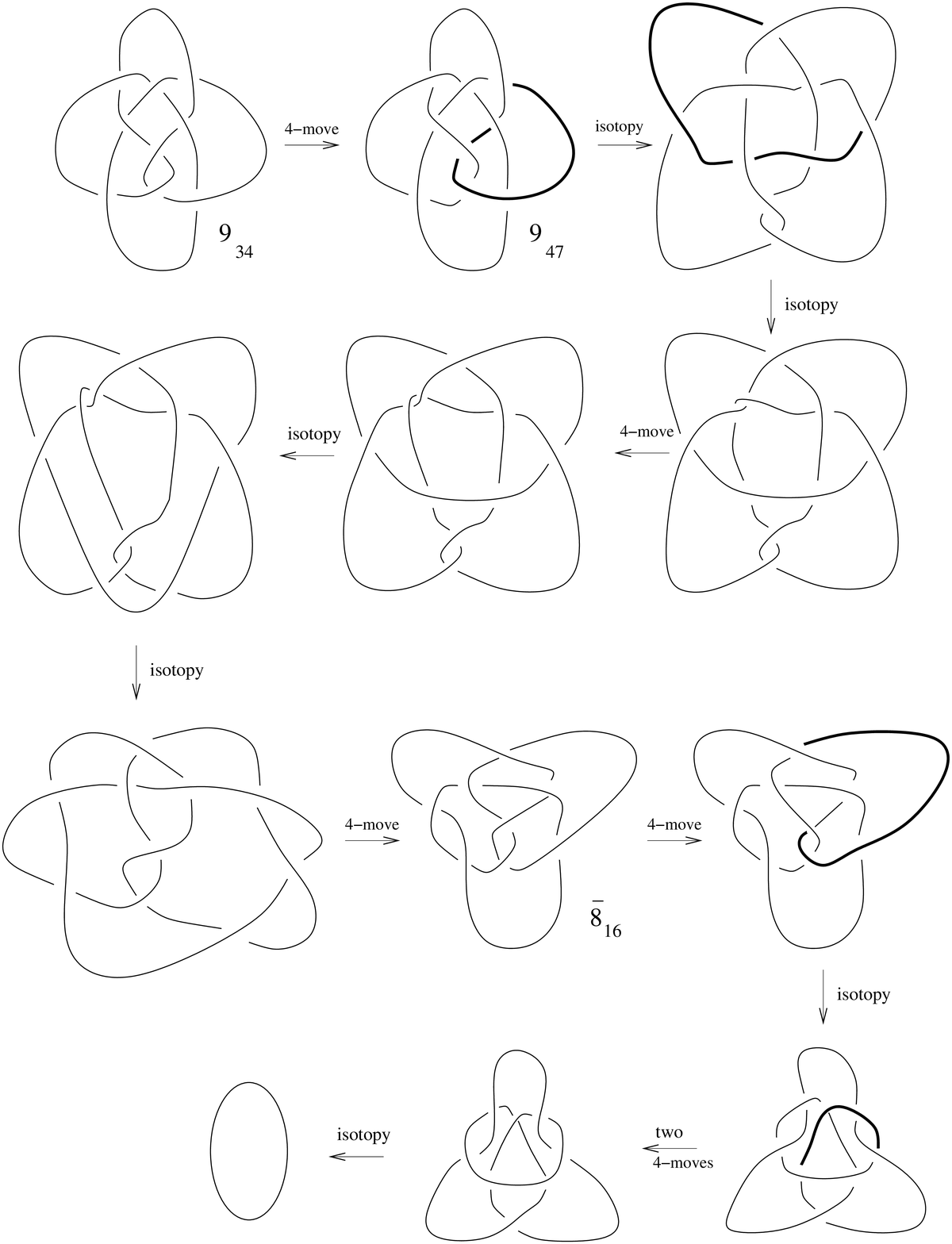,height=17.8cm}}
\ \\
\centerline{Figure 3.9; $4$-move reduction of knots $9_{34}$ and $9_{47}$}\ \\

\section{Fox $n$-coloring and its nonabelian (Burnside) generalization}\label{Section 4}

We devoted the previous section to reductions of some classes of knots (and links of two components) by 4-moves.
Now we concentrate on some invariants which were employed to try to disprove 4-move conjectures.

In this section we work with Fox colorings and their nonabelian version, Burnside groups of links.

Here we primarily follow  \cite{Prz-2,D-P-1,D-P-2}, and, an as yet not published, paper with M.D{\c a}bkowski \cite{D-P-3}.

\begin{definition}\label{Definition 4.1}

\begin{enumerate}
\item[$\mathbf{(i)}$] We say that a link $($or a tangle$)$ diagram is $k$-colored 
if every arc is colored\footnote{We call such a coloring a Fox coloring as R.Fox introduced the construction
when teaching undergraduate students at Haverford College in 1956. Fox knew well about diagrammatic description of representations 
of the fundamental group of a link complement into the dihedral group and wanted to present it in an absolutely elementary form.} 
by one of the numbers $0,1,...,k-1$ $($forming a group $\mathbb{Z}_{k})$in such a way that at each crossing the sum of the colors of the
undercrossings is equal to twice the color of the overcrossing modulo $k$; see Figure 4.1.

\item[$\mathbf{(ii)}$] The set of $k$-colorings forms an abelian group,
denoted by $Col_{k}(D)$. The cardinality of the group will be denoted by $%
col_{k}(D)$.
\end{enumerate}
\end{definition}

\centerline{\psfig{figure=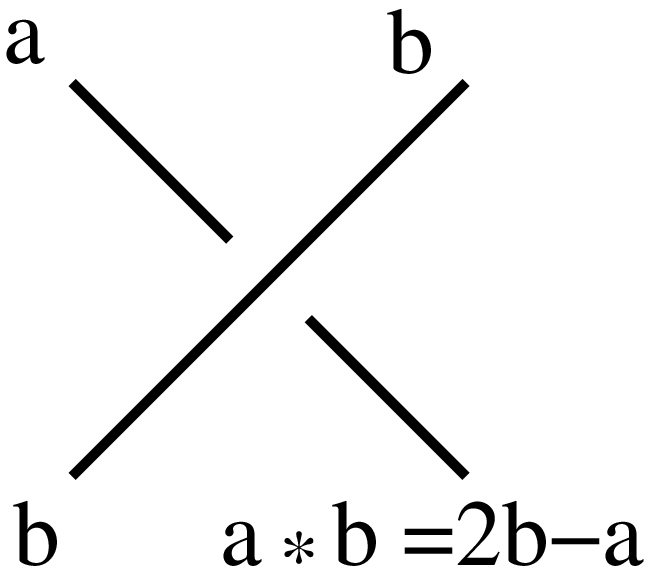,height=2.8cm}} \centerline{Figure 4.1}

\noindent One can easily show that $Col_{k}(D)$ is unchanged by Reidemeister
moves and by $k$-moves (Figure 4.2). \newline

\centerline{\psfig{figure=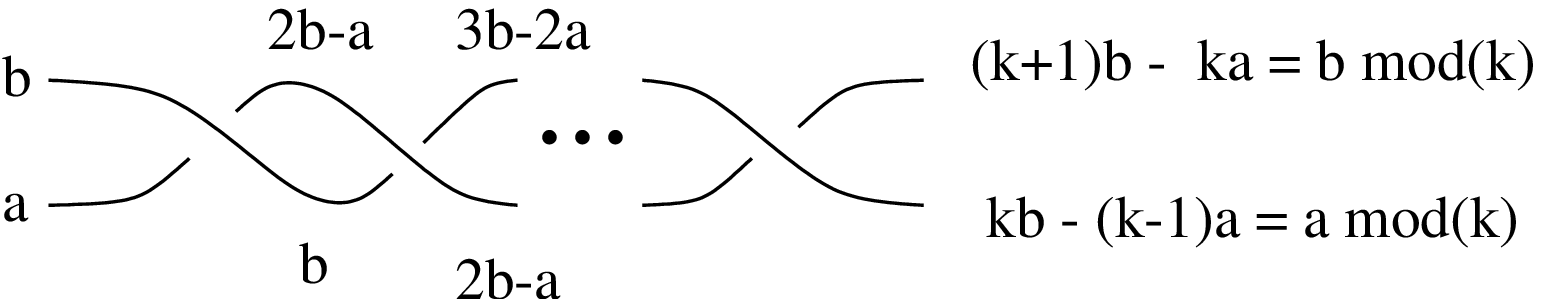,height=2.5cm}}
\centerline{Figure 4.2}

\begin{proposition}\label{Proposition 4.2} $Col_{k}(L)$ as a group is closely
related to the first homology of the double branch cover $M_{L}^{(2)}$ of $S^{3}$ along the link $L$. 
That is, the following holds:
\begin{equation*}
Col_{k}(L)=H_{1}(M_{L}^{(2)}, \mathbb{Z}_{k})\oplus \mathbb{Z}_{k}.
\end{equation*}
\end{proposition}
It is a natural thought to try to use Fox $4$-coloring on $4$-move conjectures.
However we do not gain much as $Col_4(L)$ depends only on the linking numbers matrix of $L$ taken modulo $2$.\\
In particular, for a knot $K$,  $Col_4(K)=\Z_4$ (module of trivial colorings), and for a link of $2$ components $L$, we have:
$$
Col_4(L)= 
 \left\{ \begin{array}{rl}
 \Z_4 \oplus \Z_4 &\mbox{ if $lk(L)$ is even} \\
 \Z_4 \oplus \Z_2  &\mbox{ if $lk(L)$ is odd}
       \end{array} \right.
$$

To have a simple description in a general case of an $n$-component link $L$, let $A_L$ 
be the linking matrix, with $\ell_{i,j}= lk(L_i,L_j)$. We enhance it by putting $\ell_{i,i}= -lk(L_i,L-L_i)= \sum_{j\neq i} -lk(L_i,L_j)$. 
We denote the new matrix by $M_L$. 
\begin{proposition}\label{Proposition 4.3} 
The matrix $2M_L$ with coefficients taken modulo $4$ is a presentation matrix for the $\Z_4$-module $Col_4(L)$.\footnote{The  matrix $M_L$ is related 
to Goeritz matrix in knot theory and Kirchhoff-Laplace matrix in graph theory \cite{Goe,Prz-9,Big}. This is expected due to 
Proposition \ref{Proposition 4.2}.} 
\end{proposition}
For completeness we sketch a proof here: Let $a*b\equiv 2b-a$ in $\Z_4$. We see that 
$$ a*b= 
\left\{ \begin{array}{ll}
a &\mbox{ if $b-a$ is even} \\
a+2 & \mbox{ if $b-a$ is odd}
\end{array} \right.
$$

 In particular, colors of a given component of $L$ have the same color and self-crossings are not changing colors.
Now, denote by $L_1$,$L_2$,...,$L_n$ components of $L$ choose on each of then a base point, $b_1$,...,$b_n$ respectively, 
and denote colors in base-points by $x_1$,$x_2$,...,$x_n$. Colorings at $b_1$,...,$b_n$ uniquely describe coloring of (a diagram) $L$.
Thus we can think of $x_1$,$x_2$,...,$x_n$ as of generators of a $\Z_4$ module $Col_4(L)$. We get relations by walking along 
components. That is let us walk along $L_i$ from $b_i$ in any direction till we reach $b_i$ back. Taking into account possible changing of colors,
 ignoring self-crossings, and assuming we had $k_i$ mixed undercrossings in the path, we get a relation (modulo $4$):
$$ x_i\equiv 2\sum_{j\neq i}lk(L_i,L_j)x_j +(-1)^{k_i}x_i\equiv 2\sum_{j\neq i}lk(L_i,L_j)x_j +(-1)^{lk(L_i,L-L_i)}x_i.$$
Thus:
$$2\sum_{j}\ell_{i,j}=0,$$ which is exactly the relation of the matrix $2M_L$ with entries taken modulo $4$.

The relations to homotopy skein module 
of $S^3$, \cite{H-P,Prz-4}, should be studied.

In the next subsection
we define, for an unoriented link $L$, a non-abelian
version of the group of $n$-Fox coloring, $Col_{n}(L)$. We call this group
the unreduced $n$th Burnside group of an unoriented link $L$, and denote it
by $\hat{B}_{L}(n)$. In the second subsection we show that
similarly to the observation that the group of Fox
colorings is related to the homology of the double branch cover, the
unreduced Burnside group is related to the fundamental group of the double
branch cover. 

\subsection{Burnside groups of links}\footnote{The relation to classical 
Burnside group was noticed when the computation of Figure 4.4 was made; here is a short survey of classical Burnside groups 
after the book by M. Vaughan-Lee \cite{V-Lee}:\\
Groups of finite exponents were considered for the first time by Burnside in
1902 \cite{Bur}. In particular, Burnside himself was interested in the case when $G$ is
a finitely generated group of a fixed exponent. He asked the question, known
as the Burnside Problem, whether there exist infinite and finitely generated
groups $G$ of finite exponents. \
Let $F_{r}=\langle x_{1},\,x_{2},\,\dots ,\,x_{r}\,|\,-\rangle $ be the free
group of rank $r$ and let $B(r,n)=F_{r}/N$, where $N$ is the normal subgroup
of $F_{r}$ generated by $\{g^{n}\,|\,g\in F_{r}\}$. The group $B(r,n)$ is
known as the $r$th generator Burnside group of exponent $n$. 
$B(1,n)$ is a cyclic group $\mathbb{Z} _{n}$. $B(r,2)= \mathbb{Z}_{2}^{n}$. 
Burnside proved that $B(r,3)$ is finite for all $r$ and that $B(2,4)$ is finite. 
In 1940, I. N. Sanov proved that $B(r,4)$ is finite for
all $r$ \cite{San}, and in 1958 M. Hall proved that $B(r,6)$ is finite for
all $r$ \cite{Hall}.
However, it was proved by Novikov and Adjan in 1968 \cite{N-A-1, N-A-2,
N-A-3} that $B(r,n)$ is infinite whenever $r\geq 1$ and $n$ is an odd and $n\geq
4381$ (this result was later improved by Adjan \cite{Adj}, who showed that $ B(r,n)$ is infinite 
if $r>1$ and $n$ odd and $n\geq 665$). S. Ivanov proved
that for $k\geq 48$ the group $B(2,2^{k})$ is infinite \cite{Iv}. I. G. Lys\"{e}nok found that $B(2,2^{k})$ is infinite 
for $k\geq 13$ \cite{Lys}. It is still an open problem though whether, for example, $B(2,5)$, $B(2,7)$ or $%
B(2,8)$ are infinite or finite.\ \ 
A weaker question, associated with the Burnside Problem, is known as the
Restricted Burnside Problem and it is stated as follows. Let $K$ be the
intersection of all normal subgroups $N$ of $B(r,n)$ of a finite index, and
let $R(r,n) = B(r,n)/K$.
G.Higman proved that $R(r,5)$ is finite for all $r\geq 1$ \cite{Hig}.
A.Kostrikin generalized this to all $R(r,5)$, where $p$ is a prime \cite{Ko}. 
G.Higman and P.Hall provided the reduction (based on the classification of
finite simple groups) of the Restricted Burnside Problem for an arbitrary
exponent $n$ into the Restricted Burnside Problem for prime power exponents
dividing $n$ \cite{H-H}. Finally, in 1989, E. Zelmanov solved the Restricted
Burnside Problem for prime power exponents \cite{Zel-1},\cite{Zel-2}. For this achievement he 
was awarded a Fields Medal at the International Congress of Mathematicians in Z\"urich in 1994. }

\begin{definition}\label{Definition4.4}\
\begin{enumerate}
\item[(i)] (\cite{Joy,F-R}) The associated core group $\Pi
_{D}^{(2)}$ of an unoriented link diagram $D$ is the group with generators
that correspond to arcs of the diagram and any crossing $v$ yields the
relation $r_{s}=ba^{-1}bc^{-1}$, where $b$ corresponds
to the overcrossing and $a,c$ correspond to the undercrossings at 
$ v$; see Figure 4.3.

\item[(ii)] The unreduced $n$th Burnside group of a link $L$ is
the quotient of the associated core group of the the link by its normal
subgroup generated by all relations of the form $w^{n}=1$. Succinctly:\
\begin{equation*}
\hat{B}_{L}(n)=\Pi _{L}^{(2)}/(w^{n}).
\end{equation*}

\item[(iii)] The $n$th Burnside group of a link is the quotient
of the fundamental group of the double branched cover of $S^{3}$ with the
link as the branch set divided by all relations of the form $w^{n}=1$.
Succinctly:\
\begin{equation*}
B_{L}(n)=\pi _{1}(M_{L}^{(2)})/(w^{n}).
\end{equation*}
\end{enumerate}
\end{definition}

\centerline{\psfig{figure=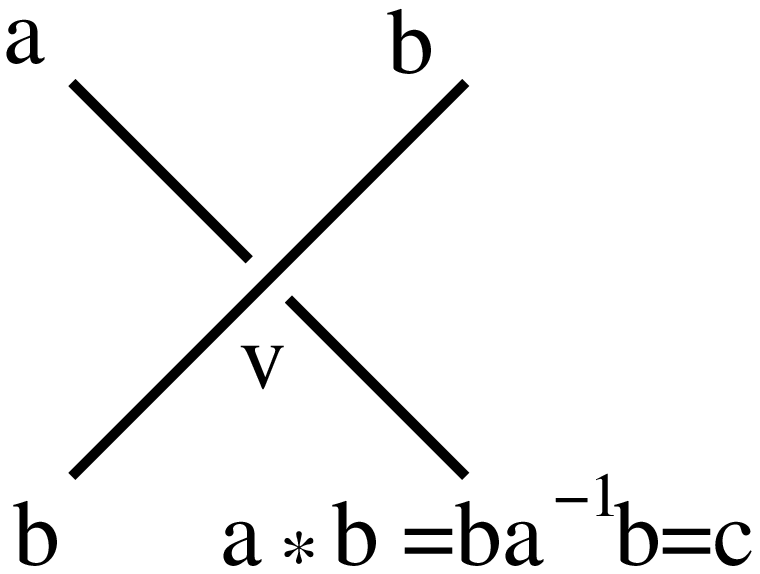,height=2.9cm}}
\centerline{Figure 4.3; core group coloring of a crossing}
\ \\
\noindent \textbf{Remark 2}\label{3.2} The associated core group of a link
was defined in \cite{F-R}. Moreover, Wada in \cite{Wad} proved that the
following holds\footnote{The presentation was also found before,
in 1976, by Victor Kobelsky a student of Oleg Viro, but this was never
published as Viro proved that the
presentation is related to the fundamental group of the double
branch covering \cite{Vi}.
See \cite{Prz-2} for an elementary proof using only Wirtinger presentation.}:
\begin{equation*}
\Pi _{D}^{(2)}=\pi _{1}(M_{L}^{(2)})\ast 
\mathbb{Z},
\end{equation*}%
In the presentation of $\Pi _{L}^{(2)}$ one relation can be dropped since it
is a consequence of others. Furthermore, if we put $y_{i}=1$ for any fixed
generator, then $\Pi _{D}^{(2)}$ reduces to $\pi _{1}(M_{L}^{(2)})$.\medskip

\noindent Notice that, abelianizing the $n$th Burnside groups of links one
obtains:
\begin{proposition}
\begin{equation*}
\hat{B}_{L}^{ab}(n)=Col_{n}(L)\text{ and }B_{L}^{ab}(n)=H_{1}(M_{L}^{(2)},
\mathbb{Z}_{n}).
\end{equation*}
\end{proposition}

The fact that $n$-moves preserve the module of Fox $n$-colorings can be 
generalized to Burnside groups of links and to rational moves (we formulate it in full generality but 
will use only for integral moves; see Figure 4.4 for the illustrative proof in this case).
\ \\ \ \\
\centerline{\psfig{figure=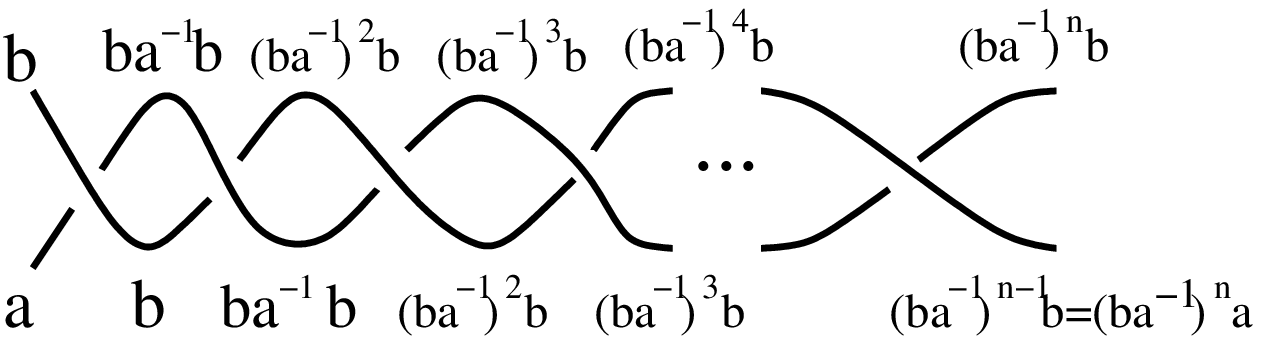,height=3.4cm}}
\ \\
\centerline{Figure 4.4;  Core group under $n$-move} 

\begin{theorem}
\label{3.3} $\hat{B}_{L}(p)$ and $B_{L}(p)$ are preserved by rational $\frac{%
p}{q}$-moves.
\end{theorem}

The method of Burnside group of links introduced in \cite{D-P-1,D-P-2,D-P-3} allowed us to settle (disprove)
conjectures of Montesinos-Nakanishi, Kawauchi, and Harikae-Nakanishi-Uchida
about 3-moves, 4-moves, and $(2,2)$-moves, respectively; in particular, links of Figure 3.7 (``half" 2-cabling of the Whitehead link), and 
Figure 3.8 (Borromean rings) are not 4-move equivalent to the trivial link of three components. We can also show
that the manifold $M^{(2)}_{9_{49}}$-- the double branched cover of $S^3$ along the link $9_{49}$ (suspected of being the smallest volume
hyperbolic oriented 3-manifold; compare \cite{M-V,Prz-7}) is not obtainable from
$(S^1 \times S^2)^k$ by a finite sequence of $q/5$-surgeries\footnote{We notice in \cite{D-P-3} that the $k/m$ rational move on a link yields
an $m/k$- surgery on the associated double branched covering of $S^3$ 
branched along the link.}. For these results we refer to \cite{D-P-1,D-P-2,D-P-3}.

Our method using Burnside groups will not produce any obstruction if the
abelianization of the Burnside group of the link is a cyclic group and the
Burnside group is finite.  In particular, the $4$th Burnside group invariant of links
cannot be used to answer Nakanishi's $4$-conjecture and Kawauchi's 4-move conjecture for links of two components.

The result follows from \cite{Mag,Ch} but we present
a short proof in our case in the next subsection.

\subsection{Limitations of Burnside groups on links} \

 Let $R_{L}(n)=B_{L}(n)/K$, where $K$ is
the intersection of all normal subgroups of $B_{L}(n)$ of finite indexes. We call  $R_{L}(n)$ the restricted Burnside group of a link.

\begin{theorem}\label{Theorem 4.7}
\label{8.1} Let $n=p^{k}$, where $p$ is a prime number and $k\geq 1$.
Assume, that the abelianization of $B_{L}(n)$, 
$ (B_{L}(n))^{(ab)}=H_{1}(M_{L}^{(2)}, \mathbb{Z}
_{n})$, is a cyclic group $(i.e.\{1\}\,{or}\, \mathbb{Z}_{n})$. 
Then the restricted Burnside group, $R_{L}(n)$, is isomorphic to $H_{1}(M_{L}^{(2)},\mathbb{Z}_{n})$. 
In particular, if $B_{L}(n)$ is finite (e.g. for $n=2,\,3,\,4)$,
then $B_{L}(n)$ is a cyclic group.
\end{theorem}

{\noindent \textit{Proof:}}\thinspace\ By a celebrated theorem of  E. Zelmanov, the
largest finite quotient of $B(r,n)$, $R(r,n)$, exists for any $n$ (see for example \cite{V-Lee}). 
Thus $R_{L}(n)$ is finite for any positive integer $n$ and for $n=p^{k}$ it is a finite $p$-group. \
Let $$B_{L}(n)=\gamma _{1}\geq \gamma _{2}=[\gamma_1,\gamma_1]\geq 
\gamma_3=[\gamma_2,\gamma_1] \geq \cdots \geq \gamma _{i}=[\gamma_{i-1},\gamma_1] \geq \cdots $$ be
the lower central series of $B_{L}(n)$. It follows from the work of Zelmanov that $R_{L}(n)=B_{L}(n)/\gamma _{i}$, 
where $B_{L}(n)=\gamma_{1} > \gamma_{2}> \cdots > \gamma_{i} = \gamma_{i+1}=...$.
If $(B_{L}(n))^{(ab)}=\{1\}$, then $ B_{L}(n)=\gamma _{2}=\dots =\gamma _{i} = \cdots$ 
and in the consequence 
$R_{L}(n)=(B_{L}(n))^{(ab)}=H_{1}(M_{L}^{(2)}, \mathbb{Z} _{n})=\{1\}$. \newline 
If $(B_{L}(n))^{(ab)}= \mathbb{Z}_{n}$, then we can choose the generators $x_{1},\dots ,x_{r}$ of $B_{L}(n)$, 
$r\geq 1$, so that $(B_{L}(n))^{(ab)}=\langle x_{1}\,|\,x_{1}^{n}\rangle $,
and $x_{j}$ can be expressed as a product of commutators (elements of $%
\gamma _{2}$) for every $j\neq 1$. Hence, for all $k,\,l\in \{1,\,2,\,\dots
,n\}$, at least one element of the pair $\{x_{l},\,x_{k}\}$ is in $\gamma
_{2}$, so each commutator $[x_{k},x_{l}]$ is a product of the third
commutators (elements of $\gamma _{3}$), thus $\gamma _{2}=\gamma _{3}$, and 
further $\gamma _{2}=\gamma _{3}=\cdots =\gamma _{i}=\cdots $,
and therefore $R_{L}(n)=B_{L}(n)/\gamma _{2}= (B_{L}(n))^{(ab)}=H_{1}(M_{L}^{(2)}, \mathbb{Z}_{n}) = {\Z}_{n}$.\newline
If $B_{L}(n)$ is finite (as it is the case of $n=2,\,3,\,4$), then 
$R_{L}(n)=B_{L}(n)$, so $B_{L}(n)$ is a cyclic group (equal to $H_{1}(M_{L}^{(2)},\Z_n))$. 
\hfill $\square $ 
\newline
 
It is proposed in \cite{D-S,BHT} that the fundamental group of a link complement, $\pi_1(S^3-L)$, be used  in place of the fundamental group 
of the double branched cover of $S^3$ branched along $L$ in the analysis of $4$-moves. In this case one have to break a $4$-move into 
$t_4$ and $\bar t_4$ and instead of Burnside condition, $w^4=1$ consider 2 conditions on $\pi_1(S^3-L)$ given by $t_4$ and $\bar t_4$ moves.
It did not produce yet new results and D{\c a}bkowski suggested that it should not work better than Burnside groups of links, 
in particular, should not disprove Nakanishi or Kawauchi 4-move conjectures.

\section{4-moves and quantum invariants}

As mentioned before, Nakanishi checked whether his conjecture can be approached using many classical knot invariants.
I checked many invariants as well when writing \cite{Prz-1} (with negative result).
Below I recall it in the case of the the Jones polynomial (starting from its unoriented version -- the Kauffman bracket polynomial).

Consider the $n$-move on the Kauffman bracket:
we start computation from the tangle $L_n= \parbox{3.9cm}{\psfig{figure=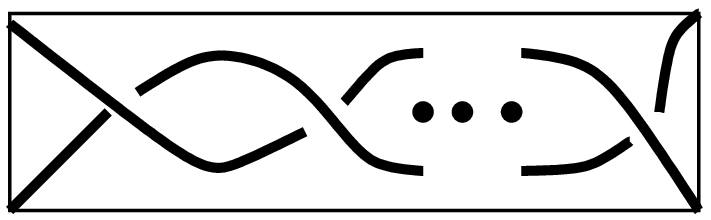,height=1.2cm}}$. 
In the Kauffman bracket skein module\footnote{See \cite{Prz-3} 
for the survey on the Kauffman bracket skein modules.} of tangles we have:
$$<L_n> = A<L_{n-1}> + A^{-1}(-A^3)^{-n+1}<L_{\infty}> = $$ 
$$A^n<L_0> + ((-1)^{1-n}A^{-3n+2}+ (-1)^{n-2}A^{-3n+6} +...+A^{n-2})<L_{\infty}>=$$
$$A^n<L_0> + A^{2-3n} \frac{A^4n+(-1)^{n-1}}{A^4+1}<L_{\infty}>=$$
Thus for $A^{4n}=(-1)^n$ (and $A^4\neq -1$) we have 
$$<L_n> = A^n<L_0>.$$ 
In particular, for a $4$-move and $(A^4)^4=1$, $A^4\neq -1$ we have
$<L_4> = A^4<L_0>$. We can put $A^{-1}=e^{\pi i/8} $ (so $A^{-2}=e^{\pi i/4}$ and $A^{-4}=e^{\pi i/2}=i$), and then:
$$<L_4> = -i<L_0>.$$
Recall that if we orient a diagram and take the writhe $w({\vec D})=\sum_{v\in Cr(D)}sgn(v)$ then
the Jones polynomial
$V_{\vec D}= (-A^3)^{-w(\vec D)}<D>$ for $t=A^{-4}$. In the case of a 4-move we put $t= A^{-4}= i$ (and ${\sqrt t}= A^{-2}=e^{\pi i/4}$).
Thus 
$$
V_{\vec L_4}= 
 \left\{ \begin{array}{rl}
 -V_{\vec L_0} &\mbox{ if arcs in $\vec L_0$ are parallel; $t_4$-move} \\
 V_{\vec L_0}  &\mbox{ if arcs in $\vec L_0$ are anti-parallel; $\bar t_4$-move}
       \end{array} \right.
$$

This should be compared with the Lickorish-Millett formula for the Jones polynomial at $t=i$ 
(precisely ${\sqrt t}= e^{\pi i/4}$), \cite{L-M,Lic}:
$$V_L(i)= \left\{ \begin{array}{rl}
(-\sqrt 2)^{com (L)-1}(-1)^{A(L)} & \mbox{ if $lk(L_j,L-L_j)$ is even for every $j$} \\
0 & \mbox{ otherwise}
\end{array} \right.
$$
In the formula $A(L)\in \{0,1\}$ is the Arf (or Robertello) invariant of oriented links\footnote{As stressed in \cite{Lic} 
the formula gives the simplest definition of the Arf invariant and explains why it is defined only if all $lk(L_j,L-L_j)$ are even.}, 
defined only if for every component $L_j$ of $L$ the 
linking number $lk(L_j,L-L_j)$ is even.

Comparing our 4-move formula with the Lickorish-Millett formula we conclude that the Jones polynomial gives very little information 
on 4-move problems, even less information than the linking number modulo 2 between individual components of a link. The best we can say is 
formulated in Proposition \ref{Proposition 5.1}.
For a trivial link, $T_n$, we have $V_{T_n}(i) = (-\sqrt 2)^{n}$, thus in the case a  link $L$ is 4-move equivalent to the 
trivial link we know that we have to use $t_4$-moves $A(L)$ (modulo 2). More generally:
\begin{proposition}\label{Proposition 5.1}. 
Assume that the links $L_1$ and $L_2$ are 4-move equivalent and their Arf invariants are defined. Then if a path from $L_1$ to $L_2$ uses
 $k$ $t_4$-moves (and any number of $\bar t_4$-moves) then $k \equiv (A(L_1) - A(L_2)) \mod 2$.
  \end{proposition}

One can similarly analyze the Homflypt and the Kauffman polynomial of two variables (using formulas from \cite{Prz-1}). 

Let us discuss shortly the case of the Kauffman polynomial of two variables $\Lambda_L(a,x)$ of regular isotopy (invariant under $R_2$ and 
$R_3$ Reidemeister move).  Recall that the
polynomial $\Lambda_L(a,x)$ is given by the skein relation $\Lambda_{L_+}(a,x) + \Lambda_{L_-}(a,x)= x(\Lambda_{L_0}(a,x)+ \Lambda_{L_{\infty}}(a,x)) $,
the framing relation, which can be expressed as saying that the positive first Reidemeister move is changing the Kauffman polynomial by $a$, and
initial condition $\Lambda_{T_1}(a,x)=1$.
We have, \cite{Prz-1}:
\begin{theorem}\label{Theorem 5.2}\
 \begin{enumerate} 
\item[(1)] If $x= p+p^{-1}$, $p^{2n}=1, p^4\neq 1$, and $a^n=p^n$, $a\neq p^{\pm 1} $,
  then the polynomial $\Lambda_L(a,x)$ is preserved by an $n$-move up to the sign (Corollary 1.15 of \cite{Prz-1}).
\item[(2)] If $p^4=-1$ (say $p=e^{\pi i/4}$) and $p^2= a^{2}=i$, $a\neq p^{\pm 1} $ (i.e. $p=-e^{\pm \pi i/4}$) then
$\Lambda_{L_4}(a,x)$ = -$\Lambda_{L_0}(a,x)$.
\end{enumerate}
\end{theorem}

We leave to readers the conclusion that the Kauffman polynomial cannot disprove 4-move conjectures. The case of colored quantum invariants 
was not, up to my knowledge, checked in all cases. Even more, invariants related to general Yang-Baxter operators (even the set theoretic ones) and 
their homology could shed light onto 4-move conjectures \cite{Jon,Tur,CES,Prz-10}.

\section{Reduction of a link $12^2_{Jab}$  of 12 crossings, $9^*.2:.2:.2$, to the Hopf link}

In November 2011 I visited Maciej Niebrzydowski at the University of Louisiana at Lafayette and we were speculating on entropic homology and 
its applications to 4-move conjectures \cite{N-P-2}. When Mietek D{\c a}bkowski visited us there, we asked him about his paper with Slavik Jablan 
and the first non-simplified link. \ 
We learned that it was an alternating 2-component link of 12 crossings, in Conway notation $9^*.2:.2:.2$ (Figure 1.1). I made then one crucial 
observation that if one performs 4-moves in three places, $m_1$, $m_2$ and $m_3$, as in Figure 6.1 one obtains a diagram representing the parallel 
2-cable of the right handed trefoil knot.  


\ \\
\centerline{\psfig{figure=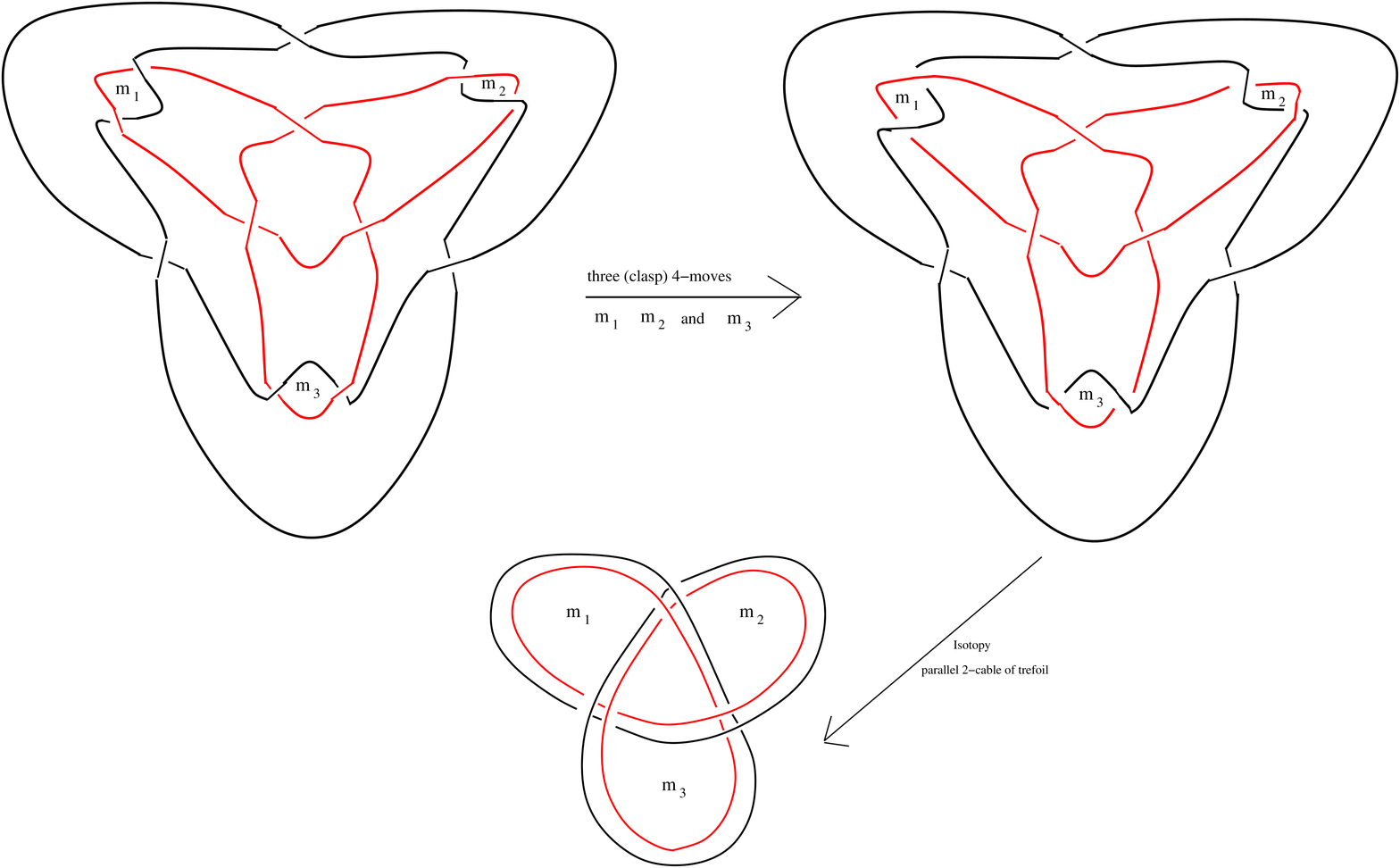,height=7.3cm}}
\centerline{Figure 6.1;  The link $12^2_{Jab}= L12a_{1388}=9^*.2:.2:.2$ changed by $4$-moves} 
\centerline{to the parallel 2-cable of the trefoil knot.}
\ \\ 

I did not notice, however, that Askitas had already given a method for the 4-move reduction of such diagrams. The reason was that I  (and authors of 
\cite{DJKS}) had the first paper of Askitas \cite{Ask-1} in which the $(2,1)$-cable of the trefoil is reduced to the trivial knot but 
we didn't have in Lafayette the second paper of Askitas \cite{Ask-2}. Only after the untimely death of Slavik in February, when I was thinking about 
writing a tribute to him  I recalled the second Askitas paper. I found it somewhere in my files (and Nikos send the picture-reduction as well, it is the 
picture we reproduce as Figure 6.4). 
And yes!, 4-move equivalence of any 2-cable of the trefoil knot to a trivial link or the Hopf link was there.
Askitas uses in a masterly way the natural reduction of Figure 3.3. He starts from an isotopy which can be interpreted as a framed Tait flype 
and then apply 4-moves (Figures 6.2 and 6.3). The move is difficult as it increases the number of crossings from 12 to 19 -- clearly 
computer couldn't find it!

Thus Figure 6.3 
shows the  framed Tait flype (in fact the red (thickened)  part is not framed in our application but the black (thinned) part is a 2-cable part):

\centerline{\psfig{figure=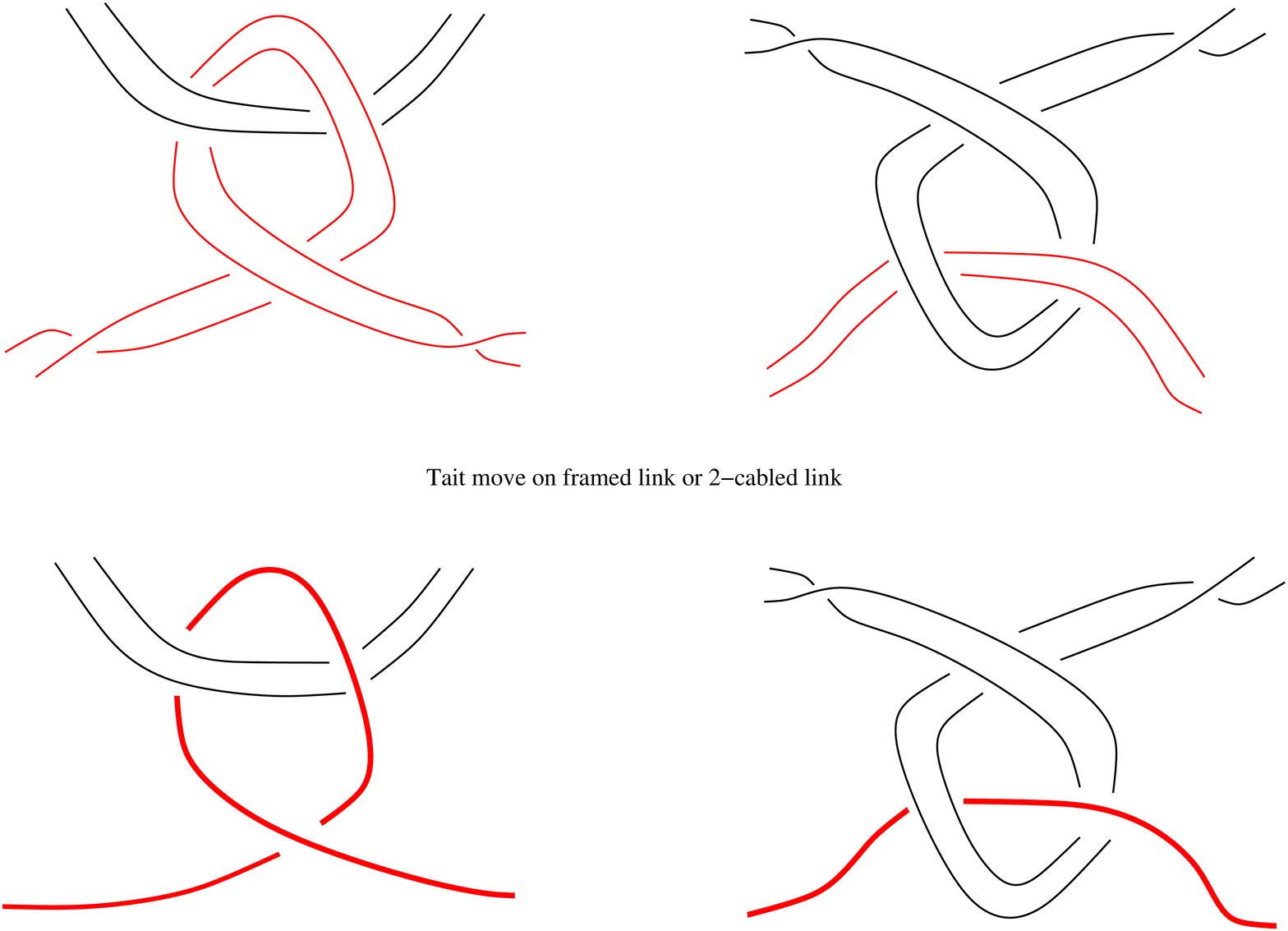,height=8.3cm}}
\centerline{Figure 6.2; Tait frame move}
\ \\
 \\
The first, difficult, step of Askitas reduction is illustrated in Figure 6.3. 

\centerline{\psfig{figure=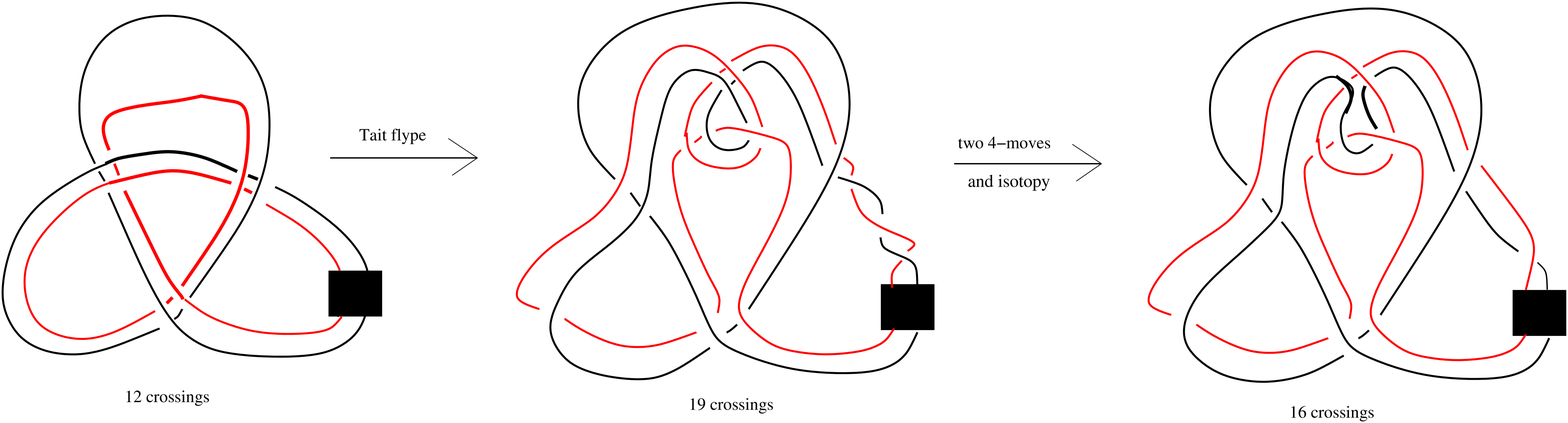,height=4.7cm}}
\ \\
\centerline{Figure 6.3;  The first steps of Askitas reduction by 4-moves}
\ \\

The full reduction is shown in Figure 6.4\footnote{I would like to thank Nikos for sending picture of his reduction.}; notice that Askitas considers
2-cable tangle of the left handed trefoil knot.

\centerline{\psfig{figure=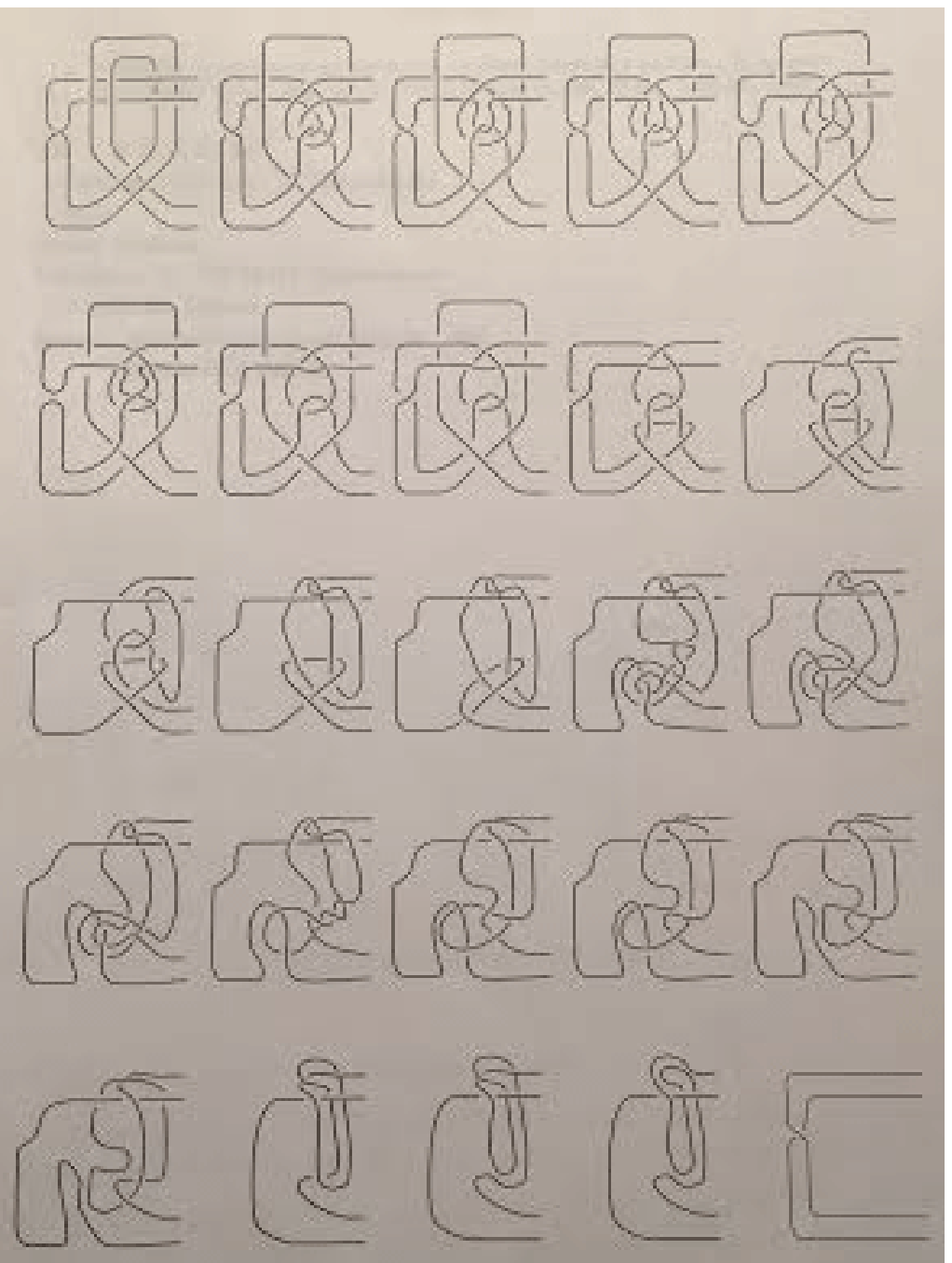,height=17.1cm}}
\ \\ \ \\
\centerline{Figure 6.4;  Tangle reduction of $2$-cable of the trefoil knot by Nikos Askitas}

\newpage
\section{Quartic skein module of links in $R^3$ and any oriented 3-manifolds}
In this section we sketch another important motivation for studying links modulo $4$-moves: we can deform moves to skein relations.

Skein relations used in skein modules of 3-dimensional manifolds are often deformations of moves on diagrams (compare \cite{Pr-Ts}).
In the case of a $4$-move its deformation should be a linear relation involving links $L_0,L_1,L_2,L_3, L_4$, and $L_{\infty}$ such that
$$x_0L_0+ x_1L_1+ x_2L_2+ x_3L_3+ x_4L_4+ x_{\infty}L_{\infty} = 0$$ for $x_0,x_1,x_2,x_3,x_4,x_{\infty}$ elements of some ring $R$, $x_0$ and $x_4$ 
invertible. It is useful to consider framed links and add framing relation $L^{(1)}=aL$ where  $L^{(1)}$ is 
obtained from a framed link $L$ by one positive twist of its framing; $a$ is an invertible element of $R$.
\begin{definition} For an oriented 3-dimensional manifold $M$ we define the quartic skein module 
$${\mathcal S}_{5,\infty}(M) = R{\mathcal L}^{fr}/(x_0L_0+ x_1L_1+ x_2L_2+ x_3L_3+ x_{4}L_{4} + x_{\infty}L_{\infty},\ L^{(1)}-aL).$$
That is we divide the free $R$-module over the ambient isotopy classes of framed unoriented links by a submodule 
generated by skein relation and framing relation. If $x_0=-x_4=1=a$, and $x_1=x_2=x_3=x_{\infty}=0$ then ${\mathcal S}_{5,\infty}(M)$ 
is a free $R$-module over the set of equivalence classes of links up to $4$-moves.
\end{definition}
We do not know much about this skein module except that it generalizes the Kauffman bracket skein module the two variable 
Kauffman skein module and the forth (cubic) skein module which was the topic of Tatsuya Tsukamoto PhD thesis, and even this only if we allow $x_4$ to 
be equal to zero. 

We can express any torus knot of type $(2,n)$, $T_{(2,n)}$, as a linear combination of trivial links and $T_{(2,2)}$ (the 
Hopf link, $H_+$).

We can also observe that because the unoriented framed left handed and right handed Hopf links are ambient isotopic ($H_-=H_+$), we have the identity
$$-(x_0+x_4)H =(a^{-1}x_1 +ax_3 +a^2x_{\infty})T_1 +x_2T_2,$$ 
where $T_n$ is the trivial framed link of $n$ components.
Thus $T_{2,n} \in span(T_1,T_2,H)$ and if $x_0+x_4$ is invertible in $R$ then $T_{2,n} \in span(T_1,T_2)$.
Similarly a careful look at Figure 2.2 should convince us that the figure eight knot, $4_1$, is in $span(T_1,T_2,H)$ and if $x_0+x_4$ is invertible, 
then $4_1\in span(T_1,T_2)$.\\
More generally (here Figure 3.2 is crucial): 
\begin{theorem}
\begin{enumerate}
\item[(1)] Every 2-algebraic tangle without a closed component is a linear combination of tangles $e_1,...,e_6$ of Figure 3.2.
\item[(2)] Every 2-algebraic knot is in $span(T_1,T_2,H)$, and if $x_0+x_4$ is invertible in $R$ then it is in $span(T_1,T_2)$ 
\end{enumerate}
\end{theorem}

\begin{proof} We directly generalize the proof of Proposition 4.3. The most involved case was considered in Figure 3.3, we solved this now 
in the quartic skein module ${\mathcal S}_{5,\infty}(Tangle)$, starting however from the simpler case of  $r(e_3*e_3)$:
\ \\ \ \\ \ \\ \ \\
$$-x_0\parbox{1.1cm}{\psfig{figure=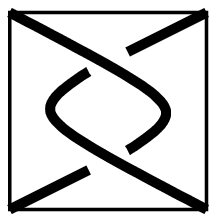,height=1.2cm}}= 
x_1 \parbox{1.1cm}{\psfig{figure=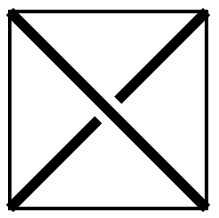,height=1.2cm}} + x_2 \parbox{1.1cm}{\psfig{figure=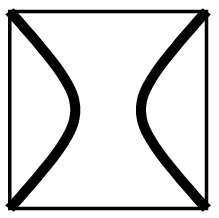,height=1.2cm}}+ 
x_3 \parbox{1.1cm}{\psfig{figure=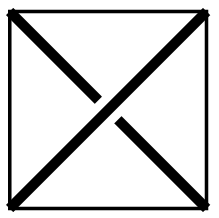,height=1.2cm}} + x_4\parbox{1.1cm}{\psfig{figure=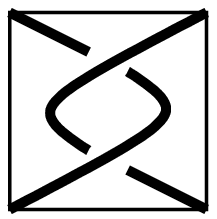,height=1.2cm}} + 
x_{\infty} \parbox{1.1cm}{\psfig{figure=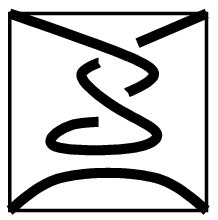,height=1.2cm}}=$$
$$x_1e_3 + x_2e_2 +x_3e_4+ x_4e_6 + a^2x_{\infty}e_1 \in span\{e_1,e_2,e_3,e_4,e_6\}$$
Now we have:
$$e_5*e_6 \stackrel{isotopy}{=}
\parbox{2.5cm}{\psfig{figure=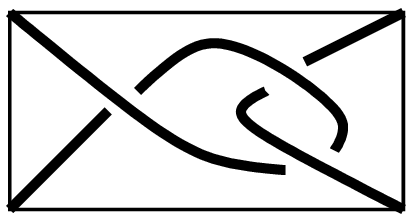,height=1.3cm}} =$$
$$x_1 \parbox{2.1cm}{\psfig{figure=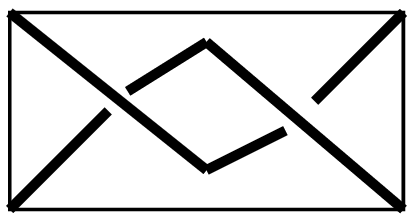,height=1.1cm}} + x_2 \parbox{2.1cm}{\psfig{figure=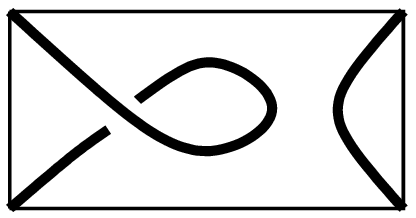,height=1.1cm}}+
x_3 \parbox{2.1cm}{\psfig{figure=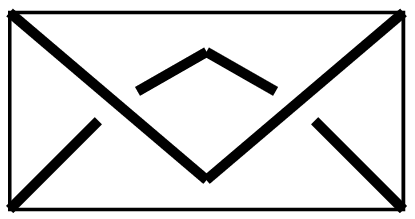,height=1.1cm}} + x_4\parbox{2.1cm}{\psfig{figure=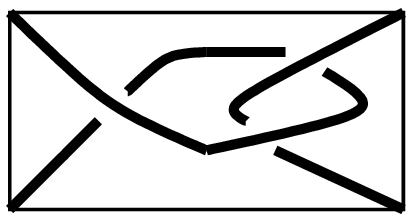,height=1.1cm}} + 
x_{\infty} \parbox{2.1cm}{\psfig{figure=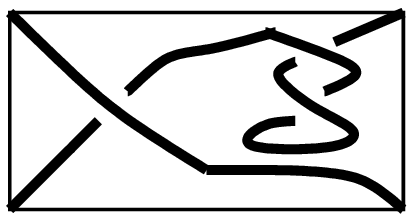,height=1.1cm}}=$$
$$x_1e_5+ a^{-1}x_2e_2+ x_3e_1+ x_4\parbox{1.1cm}{\psfig{figure=Tang-re4e4.eps,height=1.2cm}} +a^2x_{\infty}e_3\in span\{e_1,e_2,e_3,e_4,e_5,e_6\}.$$
Part (2) follows directly from Part (1).
\end{proof}
We can also generalize Theorem 4.5 to show that in the quartic skein module any two component 2-algebraic link is 
the linear combination of $T_1$, $T_2$ and the Hopf link $H$. We leave this as an exercise to the reader as
 we are straying too far from our topic, so we leave the study of the quartic skein modules for future research. I should only add that 
I did not study 3-braids in the context of the quartic skein module. Even with the Coxeter theorem, it is not checked yet that closed 3-braids are 
generated by a finite number of links in ${\mathcal S}_5(\R^3)$. We could start testing this by assuming $a_1=a_3=a_{\infty}=0$,  as in this case 
the skein relation is involving only links of the same number of components.

\section{Acknowledgements}
J.~H.~Przytycki was partially supported by the Simons Collaboration Grant-316446.

\ \\
Department of Mathematics,\\
The George Washington University,\\
Washington, DC 20052\\
e-mail: {\tt przytyck@gwu.edu},\\
and University of Gda\'nsk, Poland

\end{document}